\documentclass[preprint,3p,authoryear,a4paper]{elsarticle}
\usepackage{graphicx}
\usepackage{xcolor}
\usepackage{dsfont}
\usepackage{amsmath}
\usepackage{amssymb}
\usepackage{amsfonts}
\usepackage{amsbsy}
\usepackage{amsthm}
\usepackage{mathrsfs}
\usepackage{array}
\usepackage{booktabs} 
\usepackage{verbatim} 
\usepackage[english]{babel}
\usepackage[latin1]{inputenc}
\usepackage{float}
\usepackage{subfigure} 
\usepackage{epsfig}
\usepackage{pdfsync}
\newcounter{arclist}
\usepackage{ulem}

\newcolumntype{C}[1]{>{\centering\let\newline\\\arraybackslash\hspace{0pt}}m{#1}}
\renewcommand\appendix{\par
\setcounter{section}{0}%
\setcounter{subsection}{0}%
\setcounter{table}{0}
\setcounter{table}{0}
\setcounter{figure}{0}
\gdef\thetable{\Alph{table}}
\gdef\thefigure{\Alph{figure}}
\gdef\thesection{\Alph{section}}
\setcounter{section}{0}}

\newcommand{\E}{\ensuremath{\mathbb{E}}}
\newcommand{\Ex}{\ensuremath{\mathbb{E}^{x_1,x_2}}}

\newtheorem{theorem}{Theorem}[section]
\newtheorem{lemma}[theorem]{Lemma}

\newcommand{\bc}[1]{\textbf{#1}}
\newcommand{\BA}[1]{\bc{\textcolor{blue}{[#1 - BA]\,}}}

\newcommand{\HL}[1]{\bc{\textcolor{purple}{[#1 - \textbf{HL}]\,}}}

\usepackage[pdftex,bookmarks=true,citecolor=red,colorlinks=false,
hyperfootnotes=false]{hyperref}
\usepackage{epstopdf}

\newcommand{\pM}[1]{\begin{pmatrix}#1\end{pmatrix}}

\newcommand{\cadlag}{c\`adl\`ag}
\newcommand{\indicator}[1]{\mathbf{1}_{\{#1\}}}
\newcommand{\pa}[1]{\frac{\partial}{\partial #1}}
\newcommand{\paTo}[1]{\frac{\partial^2}{\partial #1^2}}

\makeatletter
\newcommand{\vast}{\bBigg@{3}}
\newcommand{\VaGast}{\bBigg@{4}}
\makeatother
\DeclareMathOperator{\sgn}{sgn} 
\usepackage{pgf}
\usepackage{tikz}
\usetikzlibrary{arrows,automata}
\usetikzlibrary{positioning}
\usetikzlibrary{decorations}
\tikzset{state/.style={rectangle,rounded corners,draw=black, thick,minimum height=2em,inner sep=0pt,text centered,},}
\newcommand{\cc}[1]{\multicolumn{1}{|c|}{#1}}
\usepackage{ulem}

\newcommand{\stochasticProcess}[1]{#1:=\{#1(t):t\geq 0\}} 	
\usepackage[makeroom]{cancel}

\newcommand{\subRuin}{R}
\newcommand{\subCi}{{CI}}

\newcommand{\X}{\vec{X}_\subRuin^\pi}

\newcommand{\J}{J_\subRuin}
\newcommand{\piRuin}{{\pi}}
\newcommand{\PiRuin}{\Pi_\subRuin}

\newcommand{\JJ}{J_\subCi}
\newcommand{\piCi}{\pi}
\newcommand{\FCi}{F_\subCi}
\newcommand{\parD}[1]{\frac{\partial}{\partial #1}}

\newcommand{\setenceToReferneceParameters}[1]{For parameters used, see Table \ref{tab:parameter:exp:v2} in Appendix \ref{A_B}, set no.\ #1.}

\begin{document}

\begin{frontmatter}

\title{On the surplus management of funds with assets and liabilities in presence of solvency requirements}

\author[UMelb]{Benjamin Avanzi}
\ead{b.avanzi@unimelb.edu.au}

\author[UMelb]{Ping Chen}
\ead{ping.chen@unimelb.edu.au}

\author[UCop]{Lars Frederik Brandt Henriksen\corref{cor}}
\ead{lars.brandt.henriksen@gmail.com}

\author[UNSW]{Bernard Wong}
\ead{bernard.wong@unsw.edu.au}

\cortext[cor]{
Corresponding author. Tel.: +45 29453287.
}

\address[UMelb]{Centre for Actuarial Studies, Department of Economics, University of Melbourne VIC 3010, Australia}
\address[UCop]{PFA Pension, Sundkrogsgade 4, DK-2100 Copenhagen {\O}, Denmark}
\address[UNSW]{School of Risk and Actuarial Studies, UNSW Business School, UNSW Sydney, NSW 2052, Australia}

\begin{abstract}
In this paper we consider a company whose assets and liabilities evolve according to a correlated bivariate geometric Brownian motion, such as in \citet*{GeSh03}. We determine what dividend strategy maximises the expected present value of dividends until ruin in two cases: (i) when shareholders won't cover surplus shortfalls and a solvency constraint \citep*[as in][]{Pau03} is consequently imposed, and (ii) when shareholders are always to fund any capital deficiency with capital (asset) injections. In the latter case, ruin will never occur and the objective is to maximise the difference between dividends and capital injections.

  Developing and using appropriate verification lemmas, we show that the optimal dividend strategy is, in both cases, of barrier type. Both value functions are derived in closed form. Furthermore, the barrier is defined on the ratio of assets to liabilities, which mimics some of the dividend strategies that can be observed in practice by insurance companies. Existence and uniqueness of the optimal strategies are shown. Results are illustrated.
\end{abstract}

\begin{keyword}
Optimal dividends \sep Capital injections \sep Stochastic Control \sep Regulation \sep Funding ratio \sep Solvency


MSC classes: 
93E20 \sep 
91G70 \sep 	
62P05 \sep 	
91B30 
\end{keyword}
\end{frontmatter}


\newtheorem{remark}{Remark}[section]
\numberwithin{equation}{section}

\section{Introduction}\label{S:introduction}
\subsection{Motivation and main contributions}

  The optimisation problem that is described in many modern definitions of \emph{Enterprise Risk Management} \citep*[``ERM'', see, e.g.,][]{Tay16b} is one that is closely related to the so-called \emph{stability problem} in actuarial risk theory \citep*[see, e.g.,][]{Buh70}. How to distribute dividends is one consideration \citep*[see, e.g.][for reviews of the literature in actuarial risk theory]{AlTh09,Ava09}. For instance, the \citet*{AI16} write: 

\begin{quote}An insurer's target capital policy is an integral part of its risk and capital management plans, and is likely to be used to inform dividend policy and to determine capital management triggers and mitigating actions required such as capital raising or distributions.\end{quote}

  In this paper, we consider a risky company, whose assets and liabilities follow a bivariate geometric Brownian motion with dependence. Such a model was considered by \citet*{GeSh03}, who conjectured that a barrier strategy formulated on the funding ratio of assets to liabilities was the optimal strategy. Sometimes, insurance companies use such a funding ratio based method to set their target capital; see for example \citet*{AI16,IAA16}. The conjecture of \citet*{GeSh03} was subsequently proven by \citet*{DeScGo09}. Other relevant references are \citet*[without optimality, but with finite horizon]{DeScGo06}, \citet*[without optimality, but with regime-switching dynamics]{ChYa10}, and \citet*[without optimality, but with recovery requirements]{AvHeWo18}. 

  In this paper  the company is underfunded, and ruined, if its funding ratio decreases below a certain level $\alpha_0$. This 
level would typically be 1, but it only needs {\color{black} to} be positive in this paper. For instance, a higher level than 1 could reflect the higher liquidation value of liabilities (as opposed to going-concern value). 

  In this paper, we are interested in determining the optimal dividend strategy in {\color{black} the} presence of solvency considerations. Specifically, we consider two cases. In the first case, shareholders won't be held liable further than the value of the company. In this case, a solvency constraint $\alpha_1>\alpha_0$ on the surplus level is introduced, below which no dividends are allowed to be paid. {\color{black}Note that our formulation differs from \citet*{DeScGo09}, where the authors also propose to model the assets and liabilities of a company by means of correlated geometric Brownian motion and where they also formulate an optimal control problem based on the funding ratio, but without a solvency constraint. In our paper, a dividend payment cannot bring the surplus level below a given solvency constraint.} In the second case, shareholders will always inject capital if needed. Consequently, ruin does not occur, and shareholders will maximise the difference between dividends and capital injections. Lest the problem becomes trivial, capital injections attract transaction costs. With proportional transaction costs, it turns out (unsurprisingly) that if such measures are warranted (have positive expected value), these will be made only at level $\alpha_0$ to avoid bankruptcy. A recent treatment of optimal dividends and capital injections with transaction costs is \citet{LiLi20}, although they considered a pure diffusion rather than a ratio of two correlated geometric Brownian motions as in this paper.

  Both cases described above make sense in a practical setting. Trigger points are routinely used to define different stress situations (e.g., exceeding target vs below target). Furthermore, resulting actions from management may include the injection of capital, or the reduction of dividends \citep*[see, e.g.,][]{AI16}; see also \citet*{AvHeWo18}.

  The paper is structured as follows. Section \ref{S:SolvencyConstraint} establishes the optimal dividend strategy, if ruin occurs as soon as liabilities are worth less than $\alpha_0$ assets and in presence of a solvency constraint $\alpha_1>\alpha_0$. Section \ref{S:capitalInjections} considers the case where capital injections are always made at $\alpha_0$ to prevent ruin. Numerical illustrations are presented in Section \ref{S_NumIll}.

\begin{remark}
Note that we do not consider a solvency constraint $\alpha_1$ when capital injections are forced. This is mainly because the shareholders bear the entire responsibility of a shortfall, and then there is no point in restricting the amount of dividends they can distribute from a regulatory point of view. 

That being said, the company could decide to self-inflict this constraint, as it can significantly improve the stability of the outcomes for minor cost in expectation, as discussed in Section \ref{S::simple:discussion} \citep*[see also][]{AvWo12,AvHeWo18}. One could even imagine optimising this parameter, subject to a maximum coefficient of variation of the present value of dividends (when considered as a random variable). We conjecture that the optimal dividend strategy would remain of barrier type.
\end{remark}

\subsection{The bivariate asset and liability process}
Let {\color{black} $(\Omega,\{\mathcal{F}_{t};t\geq0\},\mathbb{P})$ be a  filtrated} probability space as usually defined. We consider a risky company, whose assets $\stochasticProcess{X_1}$ and liabilities $\stochasticProcess{X_2}$, respectively, follow a bivariate geometric Brownian motion with correlation $\rho\in(-1,1)$. That is, the uncontrolled processes $(X_1,X_2)$ follow the dynamics
\begin{align}\begin{split}\label{E:dynamics:uncontrolled}
d\pM{X_1(t)\\X_2(t)}=&\pM{\mu_A&0\\0&\mu_L}\pM{X_1(t)\\X_2(t)}dt+\pM{\sigma_AX_1(t)&0\\\rho\sigma_LX_2(t)&\sqrt{1-\rho^2}\sigma_LX_2(t)}d\pM{W_1(t)\\W_2(t)},
\end{split}\end{align}
where $X_1(0)=A_0>0$, $X_2(0)=L_0>0$, $\sigma_A>0$, $\sigma_L>0$, and where $\stochasticProcess{W_1}$ and $\stochasticProcess{W_2}$ are standard Brownian motions. Furthermore, we define the funding ratio process $\stochasticProcess{Y}$ with
\begin{align*}
Y(t)=\frac{X_1(t)}{X_2(t)},\quad t\geq 0.
\end{align*}
This is the same model as in \citet*{GeSh03}.

\section{The optimal dividend strategy under a solvency constraint}\label{S:SolvencyConstraint}

\subsection{Model formulation} \label{S_modform2}

In this section, we discuss the case where we are allowed to distribute assets according to a strategy $\pi$, i.e. 
\begin{equation}
X_1^\pi(t)=X_1(t)-D^\pi(t),
\end{equation}
where $\stochasticProcess{D^\pi}$, the aggregate dividend process, is a \cadlag, non-decreasing process which is adapted to {\color{black} $\{\mathcal{F}_{t};t\geq0\}$}, the filtration generated by the bivariate process $(X_1,X_2)$. The modified funding ratio (for a strategy $\pi$) is then defined as $\stochasticProcess{Y^\pi}$ given by
\begin{align}
Y^\pi(t)=\frac{X_1^\pi(t)}{X_2(t)}.
\end{align}

We now introduce the ruin level $\alpha_0$ such that the company is ruined whenever its funding ratio (asset divided by liability) downcrosses $\alpha_0$. In this case, the time of ruin, $\tau_{\alpha_0}^\pi$, is defined as the first time the modified funding ratio $Y^\pi$ reaches $\alpha_0$, i.e. 
\begin{align}
\tau_{\alpha_0}^{\piRuin}:=\inf\left\{t\geq 0:Y^{\piRuin}(t)\, \leq\, {\alpha_0}\right\}
\end{align}
with the convention that $\inf\{\emptyset\}=\infty$. For convenience, we define (the funding ratio set) $L:2^{\mathbb{R}_+}\rightarrow 2^{\mathbb{R}^2_+}$ with
\begin{equation}
L(B)=\left\{(x,y):\frac{x}{y}\in B\right\}.
\end{equation}
For a set $B$, we use the notation $2^B$ for its power set (the set of all subsets of $B$), i.e. $2^B:=\{X:X\subseteq B\}$).
In this sense, we can see that ruin occurs at the first time the process $(X_1^\pi,X_2)$ is outside $L((\alpha_0,\infty))$.

We require that 
\begin{equation}
\frac{A_0}{L_0}\, \geq\, \alpha_0,
\end{equation}
or equivalently $(X_1(0),X_2(0))\in L([\alpha_{0},\infty))$. While we need to consider $A_0/L_0=\alpha_0$ as a boundary case, it would not make sense to start at a lower level.

In this sense, the bivariate process after payments of dividends is given by
\begin{align*}
\begin{split}
d\pM{X_1^\piRuin(t)\\X_2(t)}=&\pM{\mu_A&0\\0&\mu_L}\pM{X_1^\piRuin(t)\\X_2(t)}dt+\pM{\sigma_AX_1^\piRuin(t)&0\\\rho\sigma_LX_2(t)&\sqrt{1-\rho^2}\sigma_LX_2(t)}d\pM{W_1(t)\\W_2(t)}-d\pM{D^{\piRuin}(t)\\0}.
\end{split}
\end{align*}
A sample path of $(X^{\piRuin}_1,X_2)$ is illustrated in Figure \ref{F:intro:plot}. Note that the subscript ``$\subRuin$'' indicates that we consider a model with solvency constraint where ruin is possible. Likewise, we will throughout the paper use the subscript ``$\subCi$'' in Section 3 where ruin is prevented by capital injections. 

\begin{figure}[htb]
\begin{center}
\includegraphics[width=0.6\textwidth]{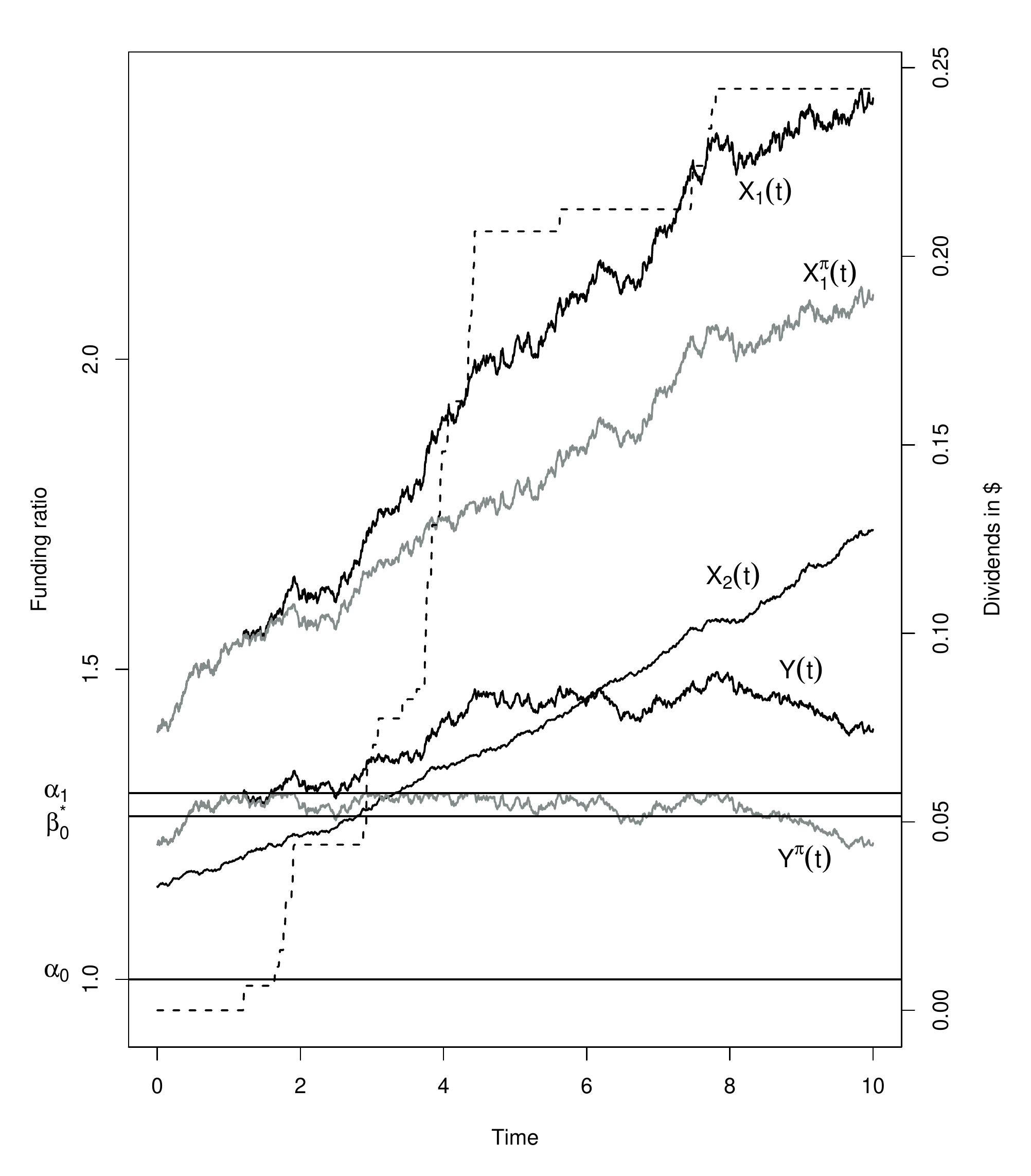}
 \caption{Figures illustrating the model. The uncontrolled processes are in black and the controlled processes are in grey. The dotted lines depict the undiscounted, aggregate payment process.}
 \label{F:intro:plot}
\end{center} 
\end{figure}

For a given (dividend) strategy $\pi$, we measure its performance by its expected present value. That is, if we denote the discount rate $\delta>0$ and the shifted operator $\mathbb{E}^{x_1,x_2}[\cdot]:=\mathbb{E}[\cdot|X_1(0)=x_1,X_2(0)=x_2]$, then the performance function is given by 
\begin{align}
\J(x_1,x_2;\piRuin)&=
\E^{x_1,x_2}\left[\int_0^{\tau_{\alpha_0}^{\piRuin}}e^{-\delta s}dD^{\piRuin}(s)\right].
\end{align}

In this section we consider a solvency constraint level $\alpha_1>\alpha_0$ such that after paying the dividend, the surplus level can never be less than $\alpha_1$. This is equivalent to the constraint introduced in \citet*[in a univariate, pure diffusion setting]{Pau03}. In mathematical terms, this constraint translates into the condition
\begin{align}\begin{split}\label{requirement:payments}
\int_0^{\ \tau_{\alpha_0}^{\piRuin}}\indicator{Y^{\piRuin}(s)<{\alpha_1}}dD^{\piRuin}(s)=0.
\end{split}\end{align}
In this sense, a strategy $\pi$ is said to be admissible if $D^{\piRuin}$ is a non-decreasing, \cadlag, {\color{black} $\{\mathcal{F}_{t};t\geq0\}$}-adapted process with $D^{\piRuin}(0-)=0$, and (\ref{requirement:payments}) satisfied. We denote $\PiRuin$ as the set of admissible strategies.

Denote the optimal value function 
\begin{equation}\label{s2value}
J_R^*(x_1,x_2):=\sup_{\piRuin\in\PiRuin}\J(x_1,x_2;\piRuin).
\end{equation}
We are interested in finding an optimal strategy $\pi^*$ such that
\begin{align}\label{eqn:for:subscript}
J_R(x_1,x_2;\piRuin^*)&=J_R^*(x_1,x_2).
\end{align}

Finally, for the model to make sense, we impose the profitable condition
\begin{align}\begin{split}\label{drift:assumption:2}
\mu_A>\mu_L
\end{split}\end{align}
and require \citep*[as in][]{GeSh03}
\begin{align}\begin{split}\label{drift:assumption}
\delta>\mu_A
\end{split}\end{align}
such that the optimal value function $J_R^*$ is finite.

\subsection{Verification lemma}
We define the extended generator $\mathscr{A}$ for the bivariate process $(X_1,X_2)$ as 
\begin{align}\begin{split}\label{E:generator:X}
\mathscr{A}f(x_1,x_2)=&\mu_Ax_1\pa{x_1}f(x_1,x_2)+\mu_Lx_2\pa{x_2}f(x_1,x_2)+\frac{1}{2}\sigma_A^2x_1^2\paTo{x_1}f(x_1,x_2)\\
&+\frac{1}{2}\sigma_L^2x_2^2\paTo{x_2}f(x_1,x_2)+\rho\sigma_A\sigma_Lx_1x_2\frac{\partial^2}{\partial x_1\partial x_2}f(x_1,x_2),
\end{split}\end{align}
for a function $f$ and $(x_1,x_2)$ such that the above is well-defined. Now, under the model described in Section \ref{S_modform2}, one sufficient condition for a strategy $\pi\in\PiRuin$ to be optimal is given by the following lemma. 

Throughout the paper, we use $\mathscr{C}^1$ to denote the class of continuously differentiable functions, and $\mathscr{C}^2$ for the class with  the first and second derivative of the function both {\color{black}existing and continuous}.

\begin{lemma}\label{verlemma2}
	For a strategy $\tilde{\pi}\in\PiRuin$, denote its value function $H(x_1,x_2):=J_R(x_1,x_2;\tilde{\pi})$. Suppose  
	$H$ satisfies the following {\color{black} six} conditions
	\begin{enumerate}
		\item $H(x_1,x_2)\geq 0$ on $L([\alpha_0,\infty))$,
		\item $H(x_1,x_2)\in\mathscr{C}^1(L([\alpha_0,\infty))\cap\mathscr{C}^2(L([\alpha_0,\infty)\backslash L(\{\alpha_1\}))$,
		\item For each $n\in\mathbb{N}$, the following holds: On $L([\alpha_0+\frac{1}{n},\alpha_0+n))$, all partial derivatives of $H$ are bounded, i.e. there is a positive number $K$ such that  $$\max\Big(\left| \frac{\partial}{\partial x_1}H(x_1,x_2)\right|,\left| \frac{\partial}{\partial x_2}H(x_1,x_2)\right|\Big)\leq K \text{ for all } (x_1,x_2)\in L([\alpha_0,\infty)),$$ 
		\item On $L([\alpha_1,\infty))$, $H$ satisfies 
		\begin{equation}
		\parD{x_1}H(x_1,x_2)\geq 1,\label{vcon1}
		\end{equation}
		\item On $L((\alpha_0,\alpha_1))$, $H$ satisfies 
		\begin{equation}
		(\mathscr{A}-\delta)H(x_1,x_2)=0,\label{vcon2}
		\end{equation}
		\item On $L((\alpha_1,\infty))$, $H$ satisfies 
		\begin{equation}
		(\mathscr{A}-\delta)H(x_1,x_2)\leq 0.\label{vcon3}
		\end{equation}
		
	\end{enumerate}
	Then $\tilde{\pi}$ is optimal, that is, $\tilde{\pi}=\pi^*$, and $J_R(x_1,x_2;\pi^*)=J_R^*(x_1,x_2)$ for all $(x_1,x_2)\in L([\alpha_0,\infty))$.
	
	\end{lemma}

\begin{proof}
According to the definition of optimal value function in (\ref{s2value}), we have $H(x_1,x_2)\leq J^{*}_{R}(x_1,x_2)$. 
	Then it suffices to show that $$H(x_1,x_2)\geq J_R(x_1,x_2;\pi)$$ for any strategy $\pi\in\PiRuin$.

	For any $\pi\in\PiRuin$, we denote the family of sequential 
	stopping times $\{T_n,n\in\mathbb{N}\}$ when the funding ratio of the modified process is outside the interval $[\alpha_0+\frac{1}{n},\alpha_0+n)$, i.e.
	$$T_n:=\inf\{t\geq 0:Y^\pi(t)\notin[\alpha_0+\frac{1}{n},\alpha_0+n)\}.$$
	
	First, from (\ref{vcon1}), for any $(x_1,x_2)\in L([\alpha_1,\infty))$ and $d\geq 0$, we have
	$$H(x_1+d,x_2)-H(x_1,x_2)=\int_{x_1}^{x_1+d}\parD{x}H(x,x_2)dx\geq \int_{x_1}^{x_1+d}dx=d,$$which gives
	\begin{equation}\label{vcon6}
	H(x_1+d,x_2)\geq H(x_1,x_2)+d.
	\end{equation}
	For $(x_1,x_2)\in L([\alpha_1,\infty))$,  it is obvious to have $(x_1+d,x_2)\in L([\alpha_1,\infty))$. For any dividend strategy $\pi$ with $D^\pi(0)=d>0$, we revise $\pi$ by removing  its dividend distribution at $t=0$ without changing its distributions at other time points, then as revealed by (\ref{vcon6}), the resulting value function associated with $(x_1+d,x_2)$ is non-decreasing. Or, equivalently, for any $(x_1-d, x_2)\in L([\alpha_1,\infty))$, we obtain
	$$H(x_1-d,x_2)+d\leq H(x_1,x_2).$$

	Then it suffices to consider the case with the dividend distribution at time 0 is $D^\pi(0)=0$. Now, we shall note that $H$ is not $\mathscr{C}^2$ over the entire range $[\alpha_0,\infty)$ and therefore we shall use the generalization of It\^{o}'s formula, the so-called Meyer-It\^{o} Formula \citep*[e.g. Theorem 70 in Chaper IV in][]{Pro05}. As pointed out by \citet*{Pes05}, $H(x_1,x_2)\in\mathscr{C}^1$ is enough to exclude the local time at $L(\{\alpha_1\})$ so that with Condition 2 in Lemma 2.1, we can proceed with the It\^{o}'s lemma in its standard form.

	On the event $\{X_1(0)=x_1, X_2(0)=x_2\}$, by using It\^{o}'s lemma for the stopped jump diffusion processes $\{e^{-\delta (t\wedge T_n)}H(X^{\piRuin}_1(t\wedge T_n),X_2(t\wedge T_n))\}$, we get that
	\begin{align*}
	&e^{-\delta (t\wedge T_n)}H(X_1^\piRuin(t\wedge T_n),X_2(t\wedge T_n))\\
	=~&H(x_1,x_2)-\int_{0-}^{t\wedge T_n} \delta e^{-\delta s}H(X_1^\piRuin(s-),X_2(s-))ds\\
	&+\int_{0-}^{t\wedge T_n}e^{-\delta s}\frac{\partial}{\partial x_1}H(X_1^\piRuin(s-),X_2(s-))dX_1^{{\piRuin},c}(s)\\
	&+\int_{0-}^{t\wedge T_n}e^{-\delta s}\frac{\partial}{\partial x_2}H(X_1^\piRuin(s-),X_2(s-))dX_2(s)\\
	&+\int_{0-}^{t\wedge T_n}e^{-\delta s}\frac{\sigma_A^2{(X_1^\piRuin(s-))}^2 }{2}\frac{\partial^2}{\partial x_1^2}H(X_1^\piRuin(s-),X_2(s-))\indicator{(X^\pi_1(s-),X_2(s-))\notin L(\{\alpha_1\})}ds\\
	&+\int_{0-}^{t\wedge T_n}e^{-\delta s}\frac{\sigma_L^2{(X_2(s-))}^2 }{2}\frac{\partial^2}{\partial x_2^2}H(X_1^\piRuin(s-),X_2(s-))\indicator{(X^\pi_1(s-),X_2(s-))\notin L(\{\alpha_1\})}ds\\
	&+\int_{0-}^{t\wedge T_n}e^{-\delta s}\rho\sigma_A\sigma_L X_1^\piRuin(s-) X_2(s-)\frac{\partial^2}{\partial x_1\partial x_2}H(X_1^\piRuin(s-),X_2(s-))\indicator{(X^\pi_1(s-),X_2(s-))\notin L(\{\alpha_1\})}ds\\
	&+\sum_{s\leq t\wedge T_n}
	e^{-\delta s}\big[H(X_1^\piRuin(s),X_2(s))-H(X_1^\piRuin(s-),X_2(s-))\big],
	\end{align*}
	where $X^{\pi,c}_1$ is the continuous part of $X^\pi_1$ and we have used the fact that the discontinuity comes from the dividends ($-\Delta X^\pi_1$).
	
	Furthermore, using $dX_1^{{\piRuin},c}(s)=\mu_A X_1^{{\piRuin},c}(s) ds + \sigma_A X_1^{{\piRuin},c}(s) dW_1(s) - dD^{{\piRuin},c}(s)$ {\color{black} where $D^{\pi,c}$ is the continuous part of $D^{\pi}$}, and similarly for $X_2(s)$, by collecting terms, we have
	\begin{align}\begin{split}\label{verp2}
	&e^{-\delta (t\wedge T_n)}H(X_1^\piRuin(t\wedge T_n),X_2(t\wedge T_n))\\=~&H(x_1,x_2)+ \int_{0-}^{t\wedge T_n} e^{-\delta s}(\mathscr{A}-\delta) H(X_1^\piRuin(s-),X_2(s-))\indicator{(X^\pi_1(s-),X_2(s-))\notin L(\{\alpha_1\})}ds\\
	&+\int_{0-}^{t\wedge T_n}e^{-\delta s}\sigma_A X_1^\piRuin(s-) \frac{\partial}{\partial x_1}H(X_1^\piRuin(s-),X_2(s-))dW_1(s)\\
	&+\int_{0-}^{t\wedge T_n}e^{-\delta s}\rho\sigma_L X_2(s-) \frac{\partial}{\partial x_2}H(X_1^\piRuin(s-),X_2(s-))dW_1(s)\\
	&+\int_{0-}^{t\wedge T_n}e^{-\delta s}\sqrt{1-\rho^2}\sigma_L X_2(s-) \frac{\partial}{\partial x_2}H(X_1^\piRuin(s-),X_2(s-))dW_2(s)\\
	&-\int_{0-}^{t\wedge T_n}e^{-\delta s} \frac{\partial}{\partial x_1}H(X_1^\piRuin(s-),X_2(s-))dD^{{\piRuin},c}(s)\\
	&+\sum_{s\leq t\wedge T_n}
	e^{-\delta s}\big[H(X_1^\piRuin(s),X_2(s))-H(X_1^\piRuin(s-),X_2(s-))\big].
	\end{split}\end{align}
	We denote the sum of the integrals in the second, the third and the fourth line of the right-hand side of (\ref{verp2}) by $M(t\wedge T_n)$. After adding and subtracting some terms, we have that 
	\begin{align}\begin{split}
	&e^{-\delta (t\wedge T_n)}H(X_1^\piRuin(t\wedge T_n),X_2(t\wedge T_n))\nonumber\\
	=~&H(x_1,x_2)+ \int_{0-}^{t\wedge T_n} e^{-\delta s}(\mathscr{A}-\delta) H(X_1^\piRuin(s-),X_2(s-))\indicator{(X^\pi_1(s-),X_2(s-))\notin L(\{\alpha_1\})}ds\nonumber\\
	&+M(t\wedge T_n)\nonumber\\
	&-\int_{0-}^{t\wedge T_n}e^{-\delta s}\left[\frac{\partial}{\partial x_1}H(X_1^\piRuin(s),X_2(s))-1\right]dD^{{\piRuin},c}(s)\nonumber\\
	&+\sum_{s\leq t\wedge T_n}
	e^{-\delta s}\big[H(X_1^\piRuin(s-)+\Delta X_1^\piRuin(s),X_2(s))-H(X_1^\piRuin(s-),X_2(s-))-\Delta X^\pi_1(s)\big]\\
	&-\left[\int_{0-}^{t\wedge T_n}e^{-\delta s}dD^{{\piRuin},c}(s)-\sum_{s\leq t\wedge T_n}e^{-\delta s}\Delta X_1^\piRuin(s)\right].
	\end{split}\end{align}
	Using the fact that $(X_1^\pi,X_2)$ are continuous except at time $t$ when dividends is paid, which changes	the $X^\pi_1$ by the amount of $-\Delta X^\pi_1(t)$, we deduce that 
	\begin{align*}
	&e^{-\delta (t\wedge T_n)}H(X_1^\piRuin(t\wedge T_n),X_2(t\wedge T_n))\nonumber\\
	=~&H(x_1,x_2)+ \int_{0-}^{t\wedge T_n} e^{-\delta s}(\mathscr{A}-\delta) H(X_1^\piRuin(s-),X_2(s-))\indicator{(X^\pi_1(s-),X_2(s-))\notin L(\{\alpha_1\})}ds\nonumber\\
	&+M(t\wedge T_n)\nonumber\\
	&-\int_{0-}^{t\wedge T_n}e^{-\delta s}\left[\frac{\partial}{\partial x_1}H(X_1^\piRuin(s),X_2(s))-1\right]dD^{{\piRuin},c}(s)\nonumber\\
	&+\sum_{s\leq t\wedge T_n}
	e^{-\delta s}\big[H(X_1^\piRuin(s-)-\Delta D^\piRuin(s),X_2(s-))-H(X_1^\piRuin(s-),X_2(s-))+\Delta D^\pi(s)\big]\\
	&-\left[\int_{0-}^{t\wedge T_n}e^{-\delta s}dD^{{\piRuin},c}(s)+\sum_{s\leq t\wedge T_n}e^{-\delta s}\Delta D^\piRuin(s)\right].
	\end{align*}
	By rearranging, we get
	\begin{align}\begin{split}
	&\left[\int_{0-}^{t\wedge T_n}e^{-\delta s}dD^{{\piRuin},c}(s)+\sum_{s\leq t\wedge T_n}e^{-\delta s}\Delta D^\piRuin(s)\right]-M(t\wedge T_n)\\
	= ~&H(x_1,x_2)\\
	&-e^{-\delta (t\wedge T_n)}H(X_1^\piRuin(t\wedge T_n),X_2(t\wedge T_n))\\
	&+\int_{0-}^{t\wedge T_n} e^{-\delta s}(\mathscr{A}-\delta) H(X_1^\piRuin(s-),X_2(s-))\indicator{((X^\pi_1(s-),X_2(s-))\notin L(\{\alpha_1\}))}ds\\
	&-\int_{0-}^{t\wedge T_n}e^{-\delta s}\left[\frac{\partial}{\partial x_1}H(X_1^\piRuin(s),X_2(s))-1\right]dD^{{\piRuin},c}(s)\\
	&+\sum_{s\leq t\wedge T_n}
	e^{-\delta s}\big[H(X_1^\piRuin(s-)-\Delta D^\piRuin(s),X_2(s-))-H(X_1^\piRuin(s-),X_2(s-))+\Delta D^\pi(s)\big].
	\end{split}
	\end{align}
	Now, by hypothesis (Conditions 1,4,5,6), (\ref{vcon6}) and recall the solvency constraint, we have
	\begin{equation*}
	\left[\int_{0-}^{t\wedge T_n}e^{-\delta s}dD^{{\piRuin},c}(s)+\sum_{s\leq t\wedge T_n}e^{-\delta s}\Delta D^\piRuin(s)\right]-M(t\wedge T_n)\leq H(x_1,x_2).
	\end{equation*}
	By taking expectation, we have
	\begin{equation*}
	\Ex\left[\int_{0-}^{t\wedge T_n}e^{-\delta s}dD^{{\piRuin}}(s)\right]-\Ex[M(t\wedge T_n)]\leq H(x_1,x_2).
	\end{equation*}
	Note that $\{M(t\wedge T_n):t\geq 0\}$ is a zero-mean martingale by Condition 3, we have
	\begin{equation*}
	\Ex\left[\int_{0-}^{t\wedge T_n}e^{-\delta s}dD^{{\piRuin}}(s)\right]\leq H(x_1,x_2).
	\end{equation*}
		
	Finally, we note that $T_n\uparrow \tau_{\alpha_{0}}^\pi$ and hence by Fatou's Lemma, we get
	\begin{align*}
	H(x_1,x_2)\geq~& \liminf_{t,n\rightarrow\infty}\Ex\Big(\int_{0-}^{t\wedge T_n}e^{-\delta s}dD^{{\piRuin}}(s)\Big)\\
	\geq ~&\Ex\Big(\liminf_{t,n\rightarrow\infty}\int_{0-}^{t\wedge T_n}e^{-\delta s}dD^{{\piRuin}}(s)\Big)\\
	=~&\Ex\Big(\int_{0-}^{\tau_{\alpha_{0}}^\pi}e^{-\delta s}dD^{{\piRuin}}(s)\Big)\\
	=~&J_R(x_1,x_2;\pi).
	\end{align*}
\end{proof}

\subsection{The value of dividends under a barrier strategy}\label{S::simple:B}

In this section, we derive the value function of a barrier strategy with an arbitrary barrier level $\beta$ (disregarding any solvency constraint). The corresponding value function is denoted by
\begin{align*}
G^\beta(x_1,x_2)&=\E^{x_1,x_2}\left[\int_{0-}^{\tau_{\alpha_0}^{\piRuin^\beta}} e^{-\delta t}dD^{\piRuin^\beta}(t)\right],
\end{align*}
where $\tau_{\alpha_0}^{\piRuin^\beta}$ is the time of ruin defined in the previous section, when a barrier strategy with barrier level $\beta$, $\piRuin^\beta$, is applied. Note that using a martingale argument, \cite{GeSh03} deduced that the value function $G^\beta$ satisfies the second condition in Lemma \ref{verlemma2}, i.e. $G^\beta\in\mathscr{C}^1(L([\alpha_0,\infty))\cap\mathscr{C}^2(L([\alpha_0,\infty)\backslash L(\{\alpha_1\}))$. 

Clearly, the value function of $G^\beta$ is given by 
\begin{align}\begin{split}\label{barrier:general}
G^\beta(x_1,x_2)=\left\{\begin{tabular}{ll}
$G(x_1,x_2;\beta),$&$ (x_1,x_2)\in L([\alpha_0,\beta])$,\\
$x_1-\beta x_2+G(\beta x_2,x_2;\beta),$&$ (x_1,x_2)\in L((\beta,\infty))$.
\end{tabular}
\right. 
\end{split}\end{align}
with $G(x_1,x_2;\beta)$ given by
\begin{align}\begin{split} \label{DE:equation:G}
(\mathscr{A}-\delta) G(x_1,x_2;\beta)=0 \hbox{ for }(x_1,x_2)\in L([\alpha_0,\beta]), \quad G({\alpha_0} x_2,x_2;\beta)=0.
\end{split}\end{align}
The differential equation (\ref{DE:equation:G}) can be obtained by the following heuristic reasoning. {\color{black}Let $\frac{x_1}{x_2}< \beta\Leftrightarrow x_1< \beta x_2$.  Consider an infinitesimal time interval of length $dt$ such that no dividend is paid during $dt$. }We get that 
\begin{align}\begin{split}\label{eqn:heuristic}
G(x_1,x_2;\beta)=&e^{-\delta dt}\E\big[G\big(x_1+\mu_Ax_1dt+\sigma_Ax_1W_1(dt),\\
&\phantom{e^{-\delta dt}\E\big[G(}x_2+\mu_Lx_2dt+ \rho\sigma_Lx_2W_1(dt) +\sqrt{1-\rho^2}\sigma_Lx_2W_2(dt);\beta\big)\big].
\end{split}\end{align}
Developing the expectation using Taylor series, subtracting $G(x_1,x_2;\beta)$ and dividing by $dt$ on both sides yields
\begin{align*}
G(x_1,x_2;\beta)\left(\frac{e^{\delta dt}-1}{dt}\right)=\mathscr{A}G(x_1,x_2;\beta). 
\end{align*}
We let $dt\rightarrow 0$ and using l'H\^opital's rule we obtain (\ref{DE:equation:G}).

  The boundary conditions for the value function, which hold for all levels of barrier $\beta$, are given by
\begin{eqnarray}
G\left({\alpha_0} x_2,x_2;\beta\right)&=&0,\label{BC:D0:B}\\
\pa{x_1}G\left(x_1,x_2;\beta\right)|_{x_1=\beta x_2-}&=&1.\label{BC:D1:B}
\end{eqnarray}
Condition (\ref{BC:D0:B}) follows directly from the definition of ruin. Condition (\ref{BC:D1:B}) is similar to \citet*[Equation (6.4)]{GeSh03} which can be obtained using the same heuristic argument.

  To find the solution to equation (\ref{DE:equation:G}), we take advantage of the fact that
$G^\beta$ (and $G$) are homogeneous functions of degree 1, which follows since {\color{black} the dynamics of both assets and liabilities are progressing proportionally to the assets and liabilities}, respectively. That is, for a given constant $\varrho$ it holds that $G^\beta(\varrho x_1, \varrho x_2)=\varrho G^\beta(x_1,x_2)$ so that the important quantity (up to a scaling factor) is the ratio $x_1/x_2$. 
The solution to the system of equations is given in the following lemma:
\begin{lemma}\label{lemma:G:simple}
The solution to differential equation (\ref{DE:equation:G}) with boundary conditions (\ref{BC:D0:B}) and (\ref{BC:D1:B}) is given by
\begin{eqnarray}\label{equation:trial:G:B}
G(x_1,x_2;\beta)
=\alpha_0 x_2\frac{\left(\frac{x_1/x_2}{\alpha_0 }\right)^{\zeta_1}-\left(\frac{x_1/x_2}{\alpha_0 }\right)^{\zeta_2}}{\zeta_1 \left(\frac{\beta}{\alpha_0}\right)^{\zeta_1-1}-\zeta_2 \left(\frac{\beta}{\alpha_0}\right)^{\zeta_2-1}},
\end{eqnarray}
where 
\begin{align}\begin{split}\label{E:sol:quad:eqn:lemma}
\tilde{\sigma}^2&=\sigma_A^2+\sigma_L^2-2\rho\sigma_A\sigma_L,\\
\zeta_1&=\frac{\frac{1}{2}\tilde{\sigma}^2-\left(\mu_A-\mu_L\right)-\sqrt{\frac{1}{4}\tilde{\sigma}^4+\left(\mu_A-\mu_L\right)^2-\tilde{\sigma}^2\left(\mu_A+\mu_L-2\delta\right)}}{\tilde{\sigma}^2}{\color{black}<0},\\
\zeta_2 &=\frac{\frac{1}{2}\tilde{\sigma}^2-\left(\mu_A-\mu_L\right)+\sqrt{\frac{1}{4}\tilde{\sigma}^4+\left(\mu_A-\mu_L\right)^2-\tilde{\sigma}^2\left(\mu_A+\mu_L-2\delta\right)}}{\tilde{\sigma}^2}{\color{black}>1}.
\end{split}\end{align}
\end{lemma}
  \emph{Proof:} See Appendix \ref{proof:lemma:G:simple}. 

\begin{remark}
Note that for $x_2=1$ this result of Lemma \ref{lemma:G:simple} is also given by \citet*[Equation (9.6)]{GeSh03}.
\end{remark}

\subsection{The optimal {\color{black}barrier} }
  We see from (\ref{equation:trial:G:B}), that only the denominator depends on $\beta$ and that both the numerator and the denominator are negative. Furthermore, it is interesting to note that, apart from a scaling factor of $x_2$, the function is only expressed in terms of ratios of $x_1$ to $x_2$. Hence, the shape of the value function is unaffected by the scale of the two processes.

Note that both numerator and denominator in \eqref{equation:trial:G:B} are negative.  We now take the derivative of the denominator and set the derivative equal to $0$ to find the maximum of the denominator (and hence the maximum of $G(\cdot;\beta)$). The resulting optimal barrier level is
\begin{eqnarray}
\beta_0^*&=&{\alpha_0}\left(\frac{\zeta_2(\zeta_2-1)}{\zeta_1(\zeta_1-1)}\right)^{\frac{1}{\zeta_1-\zeta_2}}={\alpha_0}\left(\frac{\zeta_1(\zeta_1-1)}{\zeta_2(\zeta_2-1)}\right)^{\frac{1}{\zeta_2-\zeta_1}}.\label{optimal:barrier:unconditional}
\end{eqnarray}
Because of the assumption that $\mu_A > \mu_L$, we have that $\beta_0^*>{\alpha_0}$. We obtain this result by using the representation given by (\ref{quadratic:equation}) for both the numerator and the denominator of (\ref{optimal:barrier:unconditional}),
\begin{align}\begin{split}\label{eqn:for:fraction}
\frac{\zeta_1(\zeta_1-1)}{\zeta_2(\zeta_2-1)}=\frac{\delta-\mu_L-\left(\mu_A-\mu_L\right)\zeta_1}{\delta-\mu_L-\left(\mu_A-\mu_L\right)\zeta_2}>1,
\end{split}\end{align}
and using the fact that $\frac{1}{\zeta_2-\zeta_1}>0$; see \eqref{E:sol:quad:eqn:lemma}.

 Hence, the optimal barrier level $\beta_0^*$ exists and is unique. Let $f$ denote the denominator of (\ref{equation:trial:G:B}) as a function of $\beta$. Using the assumptions (\ref{drift:assumption}) and (\ref{drift:assumption:2}) and the result (\ref{eqn:for:fraction}), we get that $f'(\alpha_0)>0$. Then because $f(\infty)=-\infty$ we know that $\beta_0^*$ maximises \eqref{equation:trial:G:B}. 
 
 Also, note that for $\delta=\mu_A$ we get that $\zeta_2=1$ which implies that $\beta_0^*\to\infty$ for $\mu_A\rightarrow \delta$ (if we become very patient ($\delta$ is getting closer to $\mu_A$) it is optimal to wait more before distributing profits) and for $\mu_A\leq \mu_L$ we get $\beta_0^*=\alpha_0$ (if the company is not profitable we should liquidate it immediately).

\begin{remark}
At  the optimal barrier $\beta_0^*$, we have 
\begin{eqnarray}
\paTo{x_1}G\left(x_1,x_2;\beta_0^*\right)|_{x_1=\beta_0^*x_2-}&=&0.\label{BC:D2:B}
\end{eqnarray}
Condition (\ref{BC:D2:B}) is obtained by taking the derivative of (\ref{BC:D1:B}) with respect to $\beta$. This gives us that (using the chain rule for partial derivatives)
\begin{align}\begin{split}\label{calc:for:condition:3}
&\pa{\beta}\left(\pa{x_1}G\left(x_1,x_2;\beta\right)|_{x_1=\beta x_2-}\right)\\
=&x_2\paTo{x_1}G\left(x_1,x_2;\beta\right)|_{x_1=\beta x_2-}+\frac{\partial^2}{\partial x_1 \partial \beta}G\left(x_1,x_2;\beta\right)|_{x_1=\beta x_2-}=\pa{\beta}1=0.
\end{split}\end{align} 
{\color{black}The term $\frac{\partial^2}{\partial x_1 \partial \beta}G\left(x_1,x_2;\beta\right)|_{x_1=\beta x_2-}$ in (\ref{calc:for:condition:3}) is claimed to be $0$ at the optimal barrier. Specifically, we see from (\ref{equation:trial:G:B}) that $G(x_1,x_2;\beta)$ and $\frac{\partial}{\partial x_1}G(x_1,x_2;\beta)$ share the same denominator with respect to $\beta$. Since the optimal barrier $\beta^*_0$ is obtained by setting $\frac{\partial}{\partial \beta}G(x_1,x_2;\beta)=0$, then $\frac{\partial^2}{\partial x_1\beta}G(x_1,x_2;\beta)$ is also 0 at $\beta^*_0$. We get (\ref{BC:D2:B}) by inserting $\beta_0^*$ and dividing the equation by $x_2$. }
\end{remark}

Finally, we denote the optimal barrier in the model with (simple) solvency constraint by $\beta_1^*$. We conjecture that the optimal barrier for the assets is given by $\beta_1^*$, where
\begin{align*}
\beta_1^*=\left\{\begin{tabular}{rl}
$\beta_0^*,$&$\beta_0^*\geq {\alpha_1}$,\\
${\alpha_1},$&$ \beta_0^*<{\alpha_1}$.
\end{tabular}
\right.
\end{align*}
With techniques similar to the ones in \citet*{DeScGo09} (who also derived $\beta_0^*$ but had no solvency constraint) we will show that $\beta_1^*$ is the optimal barrier. To make the paper self contained, we also give a proof for the case $\beta_0^*\geq {\alpha_1}$.

 \subsection{Verification of all the conditions of verification lemma}
  Next, we show that the barrier strategy with level $\beta^*_1=\max\left\{\alpha_1,\beta^*_0\right\}$ satisfies all the conditions in the verification lemma. First, we need the following lemma.
\begin{lemma} \label{lemma:update:name}
For $\alpha\geq\beta^*_0$, we have
\begin{equation}\label{alphageqbeta}
G(\alpha x_2,x_2;\alpha)\geq \frac{\alpha}{\beta^*_0}G(\beta^*_0 x_2,x_2;\beta^*_0),
\end{equation}
where $G(x_1,x_2;\beta)$ is the value function of the barrier strategy of barrier level $\beta\geq\alpha_0$ (i.e.\ without solvency constrain) with initial asset and liability values $x_1$ and $x_2$, respectively.
\end{lemma}

\begin{proof}
First, we show that
\begin{equation}\label{valuebstar0}
G(\beta^*_0 x_2,x_2;\beta^*_0) = \beta^*_0 x_2 \frac{\mu_A-\mu_L}{\color{black}\delta-\mu_L} .
\end{equation}
By definition of $\beta^*_0$ and using that $\frac{\partial{\color{black}^2}}{\partial x_1^2}G(x_1,x_2;\beta^*_0)|_{\frac{x_1}{x_2}\nearrow \beta^*_0}=0$, we have
\begin{align}\label{definition.beta0star}
\zeta_1 (\zeta_1-1){\left(\frac{\beta^*_0}{\alpha_0}\right)}^{\zeta_1} &= \zeta_2 (\zeta_2-1){\left(\frac{\beta^*_0}{\alpha_0}\right)}^{\zeta_2}\nonumber\\
\iff {\left(\frac{\beta^*_0}{\alpha_0}\right)}^{\zeta_2-\zeta_1} &= \frac{\zeta_1(\zeta_1-1)}{\zeta_2 (\zeta_2-1)}.
\end{align}
Therefore, using (\ref{equation:trial:G:B}) we have that 
\begin{align}\label{betastarvp1}
G(\beta^*_0 x_2,x_2;\beta^*_0)=~&\beta^*_0 x_2 \frac{{\left(\frac{\beta^*_0}{\alpha_0}\right)}^{\zeta_1}-{\left(\frac{\beta^*_0}{\alpha_0}\right)}^{\zeta_2}}{\zeta_1{\left(\frac{\beta^*_0}{\alpha_0}\right)}^{\zeta_1} - \zeta_2 {\left(\frac{\beta^*_0}{\alpha_0}\right)}^{\zeta_2}}\nonumber\\
=~&\beta^*_0 x_2 \frac{1-{\left(\frac{\beta^*_0}{\alpha_0}\right)}^{\zeta_2-\zeta_1}}{\zeta_1 - \zeta_2 {\left(\frac{\beta^*_0}{\alpha_0}\right)}^{\zeta_2-\zeta_1}}\nonumber\\
=~&\beta^*_0 x_2 \frac{\zeta_2(\zeta_2-1)-\zeta_1(\zeta_1-1)}{\zeta_1\zeta_2 [(\zeta_2-1)-(\zeta_1-1)]}\nonumber\\
=~&\beta^*_0 x_2 \frac{\zeta_2^2-\zeta_1^2-(\zeta_2-\zeta_1)}{\zeta_1\zeta_2(\zeta_2-\zeta_1)}\nonumber\\
=~&\beta^*_0 x_2 \frac{\zeta_2+\zeta_1-1}{\zeta_1\zeta_2}.
\end{align}
Using the definition of $\zeta_1$ and $\zeta_2$, we have 
$\zeta_1+\zeta_2-1 = \frac{2(\mu_L-\mu_A)}{{\tilde{\sigma}}^2}$ and $\zeta_1\zeta_2 = \frac{2(\mu_L-\delta)}{{\tilde{\sigma}}^2}$, we get (\ref{valuebstar0}).\\
To complete the proof, we need to show that $G(\alpha x_2;x_2;\alpha)\geq\frac{\alpha}{\beta^*_0}G(\beta^*_0 x_2,x_2;\beta^*_0) $ for $\alpha\geq\beta^*_0$. This can be done by direct computation with help from (\ref{valuebstar0}), that is
\begin{align*}
G(\alpha x_2,x_2;\alpha) =& \alpha x_2 \frac{{\left(\frac{\alpha}{\alpha_0}\right)}^{\zeta_1}-{\left(\frac{\alpha}{\alpha_0}\right)}^{\zeta_2}}{\zeta_1{\left(\frac{\alpha}{\alpha_0}\right)}^{\zeta_1} - \zeta_2 {\left(\frac{\alpha}{\alpha_0}\right)}^{\zeta_2}}\\
=& \alpha x_2 \frac{{\left(\frac{\alpha}{\alpha_0}\right)}^{\zeta_2}-{\left(\frac{\alpha}{\alpha_0}\right)}^{\zeta_1}}{\zeta_2 {\left(\frac{\alpha}{\alpha_0}\right)}^{\zeta_2}-\zeta_1{\left(\frac{\alpha}{\alpha_0}\right)}^{\zeta_1}}\\
=&\alpha x_2 \frac{{\left(\frac{\alpha}{\alpha_0}\right)}^{\zeta_2-\zeta_1}-1}{\zeta_2 {\left(\frac{\alpha}{\alpha_0}\right)}^{\zeta_2-\zeta_1}-\zeta_1}\\
=&\alpha x_2 \frac{{(\frac{\alpha}{\beta^*_0})}^{\zeta_2-\zeta_1}{\left(\frac{\beta^*_0}{\alpha_0}\right)}^{\zeta_2-\zeta_1}-1}{{(\frac{\alpha}{\beta^*_0})}^{\zeta_2-\zeta_1}\zeta_2 {\left(\frac{\beta^*_0}{\alpha_0}\right)}^{\zeta_2-\zeta_1}-\zeta_1}.
\end{align*}
By noticing that $f(x)=\frac{x-1}{x+d}$ is an increasing function for $x\geq 1$, we have
\begin{align*}\begin{split}
G(\alpha x_2,x_2;\alpha)=~&\alpha x_2 \frac{{(\frac{\alpha}{\beta^*_0})}^{\zeta_2-\zeta_1}{\left(\frac{\beta^*_0}{\alpha_0}\right)}^{\zeta_2-\zeta_1}-1}{{(\frac{\alpha}{\beta^*_0})}^{\zeta_2-\zeta_1}\zeta_2 {\left(\frac{\beta^*_0}{\alpha_0}\right)}^{\zeta_2-\zeta_1}-\zeta_1}\\
\geq~& \alpha x_2 \frac{{\left(\frac{\beta^*_0}{\alpha_0}\right)}^{\zeta_2-\zeta_1}-1}{\zeta_2 {\left(\frac{\beta^*_0}{\alpha_0}\right)}^{\zeta_2-\zeta_1}-\zeta_1}\\
=~&\frac{\alpha}{\beta^*_0}\beta^*_0x_2\frac{{\left(\frac{\beta^*_0}{\alpha_0}\right)}^{\zeta_2-\zeta_1}-1}{\zeta_2 {\left(\frac{\beta^*_0}{\alpha_0}\right)}^{\zeta_2-\zeta_1}-\zeta_1}\\
=~&\frac{\alpha}{\beta^*_0}G(\beta^*_0 x_2,x_2;\beta^*_0).
\end{split}\end{align*}
\end{proof}

  Next, we show that the value function of the barrier strategy with barrier level $\beta^*_0$ is concave in the first argument, i.e. 
\begin{equation}\label{valuebstarconcave}	
\frac{\partial{\color{black}^2}}{\partial x_1^2}G^{\beta^*_0}(x_1,x_2)\leq 0.
\end{equation}
Specifically, when $\frac{x_1}{x_2}\leq \beta^*_0$, we have
\begin{align}
\frac{\partial^2}{\partial x_1^2}G^{\beta^*_0}(x_1,x_2)=~&\frac{\partial^2}{\partial x_1^2}G(x_1,x_2;\beta^*_0)\nonumber\\
=~&\beta^*_0 x_2 \frac{\zeta_2(\zeta_2-1){\big(\frac{x_1}{\alpha_0x_2}\big)}^{\zeta_2}x_1^{-2}-\zeta_1(\zeta_1-1){\big(\frac{x_1}{\alpha_0x_2}\big)}^{\zeta_1}x_1^{-2}}{\zeta_2{\big(\frac{\beta^*_0}{\alpha_0}\big)}^{\zeta_2}-\zeta_1{\big(\frac{\beta^*_0}{\alpha_0}\big)}^{\zeta_1}}\nonumber\\
=~&\frac{\beta^*_0 x_1^{-2} x_2{\big(\frac{x_1}{\alpha_0 x_2}\big)}^{\zeta_1}}{\zeta_2{\big(\frac{\beta^*_0}{\alpha_0}\big)}^{\zeta_2}-\zeta_1{\big(\frac{\beta^*_0}{\alpha_0}\big)}^{\zeta_1}}\Big(\zeta_2(\zeta_2-1){\big(\frac{x_1}{\alpha_0x_2}\big)}^{\zeta_2-\zeta_1}-\zeta_1(\zeta_1-1)\Big)\nonumber\\
\leq~&\frac{\beta^*_0 x_1^{-2} x_2{\big(\frac{x_1}{\alpha_0 x_2}\big)}^{\zeta_1}}{\zeta_2{\big(\frac{\beta^*_0}{\alpha_0}\big)}^{\zeta_2}-\zeta_1{\big(\frac{\beta^*_0}{\alpha_0}\big)}^{\zeta_1}}\Big(\zeta_2(\zeta_2-1){\big(\frac{\beta^*_0}{\alpha_0}\big)}^{\zeta_2-\zeta_1}-\zeta_1(\zeta_1-1)\Big)\nonumber\\
=~&0\label{eqn>:conc:special:Case}
\end{align}	
by equation (\ref{definition.beta0star}). That is, by (\ref{barrier:general}) and (\ref{eqn>:conc:special:Case}), (\ref{valuebstarconcave}) holds. 

  Now we are ready to show that our candidate strategy, a barrier strategy with barrier level $\beta^*_1=\max\left\{\beta^*_0,\alpha_1\right\}$, is the optimal strategy. We proceed by showing that the value function associated to the barrier strategy satisfies all the conditions in the verification lemma. This is shown by the next lemma.
\begin{lemma}\label{lemma:satisfaction:of:condition}
	The barrier strategy with barrier level $\beta^*_1$ is admissible and its value function $G^{\beta^*_1}$ satisfies all the conditions in the verification lemma, Lemma \ref{verlemma2}.
\end{lemma}
\begin{proof}
The barrier strategy is admissible because it only brings down the asset level when the ratio is above $\beta_1^*=\max\left\{\beta^*_0,\alpha_1\right\}\geq\alpha_1$ and it does nothing when the ratio is below $\alpha_1$. By the definition of the strategy, there is never a negative contribution and therefore Condition 1 in Lemma \ref{verlemma2} is automatically satisfied. For the other conditions, we separate the cases into $\beta^*_0\geq\alpha_1$ and $\beta^*_0<\alpha_1$.

\begin{enumerate}
	\item[Case 1:]  $\beta^*_0\geq\alpha_1$. The strategy is simply the barrier strategy with barrier level $\beta^*_0$. By (\ref{barrier:general}) and (\ref{equation:trial:G:B}), Condition 2 is obvious for $G^{\beta^*_1}$ except for the second order differentiability at the barrier point $\beta_0^*$. In view of this and by the definition of $\beta^*_0$, we have
	 
	\[
	\frac{\partial{\color{black}^2}}{\partial x_1^2}G^{\beta^*_1}(x_1,x_2)|_{\frac{x_1}{x_2}\nearrow\beta^*_0}=\frac{\partial{\color{black}^2}}{\partial x_1^2}G(x_1,x_2;\beta^*_0)|_{\frac{x_1}{x_2}\nearrow\beta^*_0}=0.
	\]
	Hence, we also have
	\begin{equation*}
	\frac{\partial{\color{black}^2}}{\partial x_2^2}G^{\beta^*_1}(x_1,x_2)|_{\frac{x_1}{x_2}\nearrow\beta^*_0}={\beta^*_0}^2\frac{\partial{\color{black}^2}}{\partial x_1^2}G^{\beta^*_1}(x_1,x_2)|_{\frac{x_1}{x_2}\nearrow\beta^*_0}=0,
	\end{equation*}
	and
	\begin{equation*}
	\frac{\partial{\color{black}^2}}{\partial
		x_1\partial x_2}G^{\beta^*_1}(x_1,x_2)|_{\frac{x_1}{x_2}\nearrow\beta^*_0}={\beta^*_0}\frac{\partial{\color{black}^2}}{\partial x_1^2}G^{\beta^*_1}(x_1,x_2)|_{\frac{x_1}{x_2}\nearrow\beta^*_0}=0.
	\end{equation*}
	Therefore, $G^{\beta^*_1}\in\mathscr{C}^2(L((\alpha_0,\infty)))$ and Condition 2 is established. For $\frac{x_1}{x_2}\leq\beta^*_0$, by the definition of the strategy, we have $(\mathscr{A}-\delta)G^{\beta^*_1}(x_1,x_2)=0$.
	For $\frac{x_1}{x_2}\geq\beta^*_0$,
	\begin{align*}
	(\mathscr{A}-\delta)G^{\beta^*_1}(x_1,x_2)=~&-\delta(x_1-\beta^*_0 x_2)+ (\mathscr{A}-\delta)G^{\beta^*_0}(\beta^*_0 x_2,x_2)\\
	\leq~& 0.
	\end{align*}
	This proves (\ref{vcon2}) and (\ref{vcon3}) (Conditions 5 and 6).\\
	By using (\ref{valuebstarconcave}), $\frac{\partial}{\partial x_1}G^{\beta^*_1}(x_1,x_2)|_{\frac{x_1}{x_2}\geq\beta^*_0}=1$ and $G^{\beta^*_1}\in\mathscr{C}^2$, we have $\frac{\partial}{\partial x_1}G^{\beta^*_1}(x_1,x_2)\geq1$, which leads to (\ref{vcon1}) (Condition 4). \\
	Since $G^{\beta^*_1}$ is concave in the first argument, we have $\frac{\partial}{\partial x_1}G^{\beta^*_1}(x_1,x_2)\leq\frac{\partial}{\partial x_1}G^{\beta^*_1}(x_1,x_2)|_{\frac{x_1}{x_2}=\alpha_0} = \frac{\beta^*_0}{\alpha_0}\frac{\zeta_1-\zeta_2}{\zeta_1{\left(\frac{\beta^*_0}{\alpha_0}\right)}^{\zeta_1}-\zeta_2{\left(\frac{\beta^*_0}{\alpha_0}\right)}^{\zeta_2}}$. Combining with $\frac{\partial}{\partial x_1}G^{\beta^*_1}(x_1,x_2)\geq 1$, we have shown that $\parD{x_1}G^{\beta_1^*}(x_1,x_2)$ is bounded.\\
	Lastly, we show that $\parD{x_2}G^{\beta_1^*}(x_1,x_2)$ is bounded (and therefore establish Condition 3) by direct computation combining the two branches, i.e.
	\begin{align*}
	\left| \frac{\partial}{\partial x_2}G^{\beta^*_1}(x_1,x_2)\right|
	\leq&\max\left\{|-\beta^*_0|,\left| (1-\zeta_1)K_1{\left(\frac{x_1}{x_2}\right)}^{\zeta_1}+(1-\zeta_2)K_2{\left(\frac{x_1}{x_2}\right)}^{\zeta_2}\right| \right\}\\
	\leq&\max\left\{\beta^*_0,|(1-\zeta_1)K_1|{\left(\frac{x_1}{x_2}\right)}^{\zeta_1}+|(1-\zeta_2)K_2|{\left(\frac{x_1}{x_2}\right)}^{\zeta_2}\right\}\\
	\leq&\max\left\{\beta^*_0,|(1-\zeta_1)K_1|{\alpha_0}^{\zeta_1}+|(1-\zeta_2)K_2|{\beta^*_0}^{\zeta_2}\right\},
	\end{align*}
	where the maximum is taken over the two branches and $K_1$ and $K_2$ are pre-determined constants given by (\ref{equation:trial:G:B}). All conditions are established. 
	
	\item[Case 2:] $\alpha_1>\beta^*_0$. Since the barrier level is $\alpha_1$, Conditions 2,4 and 5 are satisfied automatically. \\
	Using similar arguments as above, i.e.\ separating the value function into 2 branches and making use of the properties that the value function (and its derivatives) is continuous and therefore bounded on one branch, equal to a constant on the other branch, we can show that Condition 3 is satisfied.\\
	For Condition 6, with the help of the second branch of (\ref{barrier:general}), (\ref{alphageqbeta}) and (\ref{valuebstar0}), for $\frac{x_1}{x_2}\geq\alpha_1$, we have
	\begin{align*}
	(\mathscr{A}-\delta)G^{\beta^*_1}(x_1,x_2)=~&\mu_A x_1 +\mu_Lx_2(-\alpha_1+\frac{\partial}{\partial x_2}G(\alpha_1x_2,x_2;\alpha_1))-\delta(x_1-\alpha_1x_2+G(\alpha_1x_2,x_2;\alpha_1))\\
	=~&\mu_A x_1 +\mu_Lx_2(-\alpha_1+\frac{G(\alpha_1x_2,x_2;\alpha_1)}{x_2})-\delta(x_1-\alpha_1x_2+G(\alpha_1x_2,x_2;\alpha_1))\\
	=~&\left(\mu_A-\delta\right)x_1+(\delta-\mu_L)\alpha_1x_2+(\mu_L-\delta)G(\alpha_1x_2,x_2;\alpha_1)\\
	\leq~&\left(\mu_A-\delta\right)\alpha_1x_2+(\delta-\mu_L)\alpha_1x_2+(\mu_L-\delta)\left(\frac{\alpha_1}{\beta^*_0}\right)G(\beta^*_0 x_2,x_2;\beta^*_0)\\
	=~&\left(\mu_A-\mu_L\right)\alpha_1x_2 +(\mu_L-\delta)\left(\frac{\alpha_1}{\beta^*_0}\right)
	\beta^*_0 x_2\frac{\mu_L-\mu_A}{\mu_L-\delta}\\
	=~&\left(\mu_A-\mu_L\right)\alpha_1x_2 +\alpha_1x_2(\mu_L-\mu_A)\\
	=~&0.
	\end{align*}
	All conditions are established. 
\end{enumerate}  
\end{proof}

As all conditions in the verification lemma (Lemma \ref{verlemma2}) are satisfied, we can conclude that $\pi^{\beta^*_1}$ is optimal in our formulation. This is restated in the following.

\begin{theorem}
	With $\beta_0^*$ defined by (\ref{E:sol:quad:eqn:lemma}) and (\ref{optimal:barrier:unconditional}), the barrier strategy with barrier level $\beta^*_1=\max(\beta_0^*,\alpha_1)$ is optimal in the presence of solvency constraint at level $\alpha_1>\alpha_0$.
\end{theorem}

\begin{proof}
The result follows from Lemmas \ref{verlemma2} and \ref{lemma:satisfaction:of:condition}.
\end{proof}

\section{When capital is injected to prevent ruin}\label{S:capitalInjections}

\subsection{The model with capital injections} \label{S_modform3}
In this section, we assume that capital injections are made in order to prevent ruin. In this case, a strategy $\pi$ consists of {\color{black} two} parts: (1) $D^\pi(t)$ to pay dividends to shareholders, and (2) $E^\pi(t)$ to rescue the company when the asset is too low. In this sense, we have 
\begin{equation}
X^\pi_1(t)=X_1(t)-D^\pi(t)+E^\pi(t),
\end{equation}
where both $D^\pi$ and $E^\pi$ are non-decreasing, \cadlag  $\;$and {\color{black} $\{\mathcal{F}_{t};t\geq0\}$}-adapted processes such that 
\begin{equation}
\frac{X^\pi_1(t)}{X_2(t)}\geq \alpha_{0},\quad t\geq 0.
\end{equation}
The set of admissible strategies is denoted as $\Pi_{CI}$. The dynamics of the bivariate process $(X_1^\pi,X_2)$ become
\begin{align*}
\begin{split}
d\pM{X_1^{\piCi}(t)\\X_2(t)}=&\pM{\mu_A&0\\0&\mu_L}\pM{X_1^{\piCi}(t)\\X_2(t)}dt+\pM{\sigma_AX_1^{\piCi}(t)&0\\\rho\sigma_LX_2(t)&\sqrt{1-\rho^2}\sigma_LX_2(t)}d\pM{W_1(t)\\W_2(t)}-d\pM{D^{\piCi}(t)-E^{\piCi}(t)\\0}.
\end{split}
\end{align*} 

Similar to the case in Section \ref{S:SolvencyConstraint}, the value function of a strategy $\pi\in\Pi_{CI}$ is given by
\begin{align}\begin{split}\label{objective:CI}
\JJ(x_1,x_2;\piCi)&=\limsup_{t\rightarrow\infty}\Ex\left[\int_0^{t}e^{-\delta s}dD^{\piCi}(s)-\int_0^{t}\kappa e^{-\delta s}dE^{\piCi}(s) \right],
\end{split}\end{align}
where $\kappa$ is a constant strictly greater than $1$ which recognises that there is a price of raising capital. 

We also denote the optimal value function as 
\begin{equation}\label{s3value}
\JJ^*(x_1,x_2)=\sup_{\pi\in\Pi_{CI}}\JJ(x_1,x_2;\pi)
\end{equation}
and we are interested in finding a strategy $\pi^*\in\Pi_{CI}$ such that 
\begin{equation}
\JJ(x_1,x_2;\pi^*)=\JJ^*(x_1,x_2),\quad \frac{x_1}{x_2}\geq \alpha_0.
\end{equation}
In addition to the above, because $\kappa>1$ we restrict the set of admissible strategies so that
\begin{equation}\label{efficient.pi}
\Delta D^\pi(t)\Delta E^\pi(t)=0,
\end{equation}
that is, the company does not raise money to pay dividends, which would not be optimal in view of (\ref{objective:CI}). Conditions \eqref{drift:assumption:2} and \eqref{drift:assumption} are also assumed to be true here.

\begin{figure}[!htb]
\begin{center}
\includegraphics[width=0.6\textwidth, clip=true, trim = 0mm 0mm 0mm 0mm]{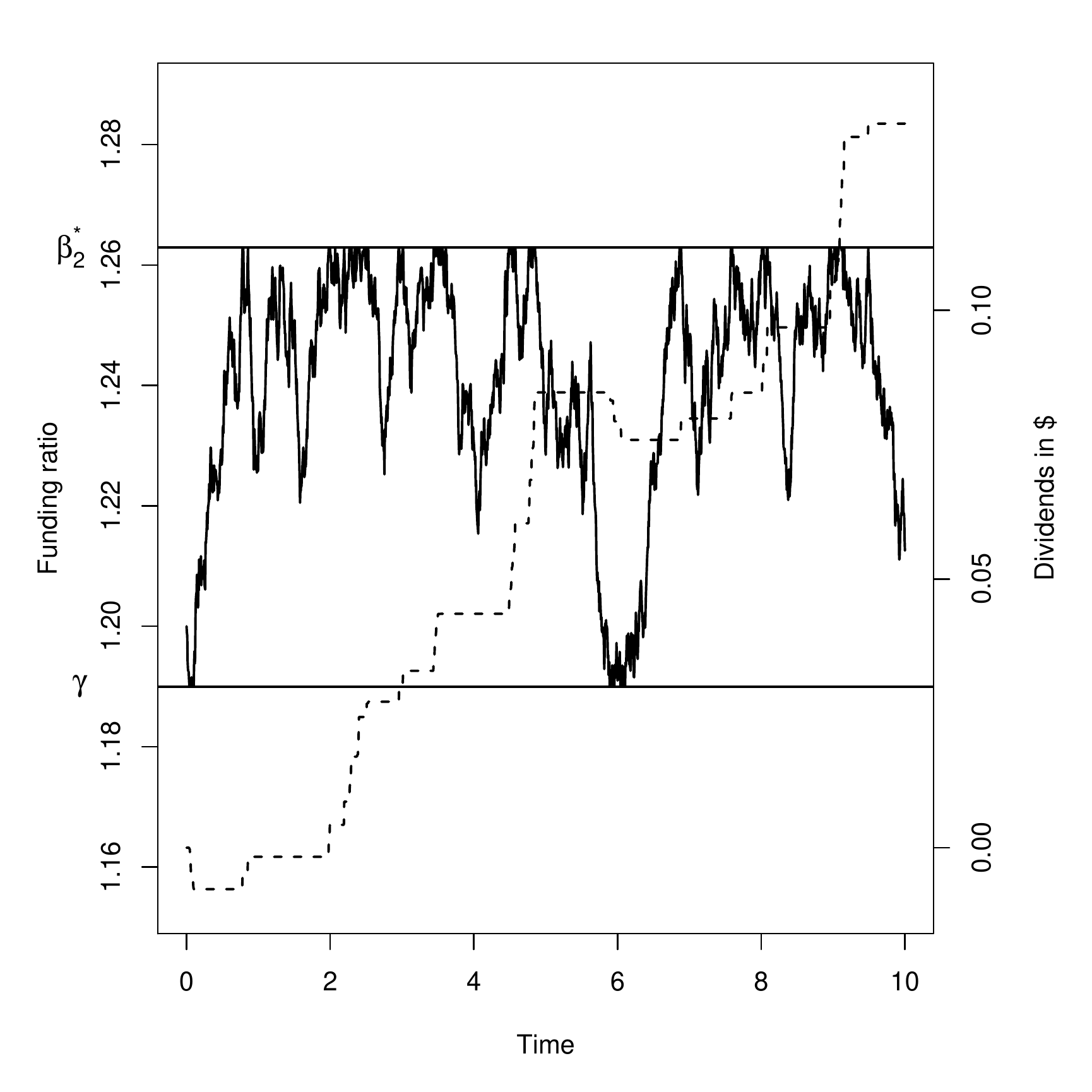}
 \caption{Simulation of the controlled funding ratio in the model where ruin must be prevented. The dotted lines are the net distributions.} 
 \label{F:CI:illustration}
\end{center}
\end{figure}

  The model is illustrated in Figure \ref{F:CI:illustration}, where the funding ratios are depicted by solid lines and net distributions (distributions, minus rescue measures, in absence of transaction costs) by dotted lines for a given sample path of the two-dimensional Brownian motion. The plot also contains the optimal upper barrier, which we find in the following sections.

\begin{remark}\label{R_forced}
In this section, capital is always injected to prevent ruin. Such capital injections are not always desirable, in that the expected present value of future dividends gained from the rescue may not compensate for the cost of the capital injection. But in some cases it could indeed be profitable. This idea goes back to \citet*[Chapter 20]{Bor74} and \citet*{Por77}. 

  Determining the optimal \emph{unconstrained} dividend and capital injection strategy in our context is outside the scope of this paper. It makes sense that capital injections would be warranted if, and only if, the associated value function $\JJ(\gamma x_2,x_2;\beta,\gamma)$ was greater than or equal to zero. This conjecture is supported by previous results in the literature \citep*[see, for instance][]{LoZe08,AvShWo11}. This is explored in Section \ref{S_rescue}.
\end{remark}

\subsection{Verification lemma}

Under the model formulation described in Section \ref{S_modform3}, one sufficient condition for a strategy $\pi\in\Pi_{CI}$ to be optimal is given by the following lemma. 

\begin{lemma}\label{verlemma3}
	
	For a strategy $\hat{\pi}\in\Pi_{CI}$, denote its value function $H(x_1,x_2):=\JJ(x_1,x_2;\hat{\pi})$. Suppose on $L(\alpha_0,\infty)$, $H$ satisfies the following 5 conditions 
 \begin{eqnarray}
	H(x_1,x_2)\in\mathscr{C}^2(L(\alpha_0,\infty)),\label{vercon0}\\
	H(x_1,x_2)&\geq& 0, \label{vercon1}\\
	(\mathscr{A}-\delta)H(x_1,x_2)&\leq& 0, \label{vercon2}\\
	1\leq\frac{\partial}{\partial x_1} H(x_1,x_2)&\leq&\kappa, \label{vercon3}\\
	\left\lvert\frac{\partial}{\partial x_2} H(x_1,x_2)\right\rvert&\leq& K_0, \label{vercon4}
	\end{eqnarray}
	where $K_0$ is a positive number.
Then $\hat{\pi}$ is optimal, that is,  $\hat{\pi} = \pi^*$, and $\JJ(x_1,x_2;\pi^*)=\JJ^*(x_1,x_2)$ for all $(x_1,x_2)\in L([\alpha_0,\infty))$.
\end{lemma}

\begin{proof}
According to the definition of value function in (\ref{s3value}), it is obvious to have $H(x_1,x_2)\leq \JJ^*(x_1,x_2)$. Then it is suffices to show that
$$H(x_1,x_2)\geq \JJ(x_1,x_2;\pi)$$
for any strategy $\pi\in\Pi_{CI}$.
First, from (\ref{vercon3}), using the mean value theorem, we can show that for any $(x_1,x_2)\in L([\alpha_0,\infty)$ and $d\geq 0$ we have
\begin{equation}
H(x_1,x_2)+d\leq H(x_1+d,x_2).\label{vercon5}
\end{equation}
On the other hand, (\ref{vercon3}) also implies that for $\epsilon\geq 0$ we have
\begin{equation}\label{vercon5b}
H(x_1-\epsilon,x_2)\geq H(x_1,x_2)-\kappa \epsilon.
\end{equation}	
Similar to Lemma 2.1, thanks to (\ref{vercon5}) and (\ref{vercon5b}) it suffices to consider the case when $D^\pi(0)=E^\pi(0)=0$.

Condition on the event $\{(X_1(0),X_2(0))=(x_1,x_2)\}$, by using It\^{o}'s lemma for the process $\{e^{-\delta t}H(X^\pi_1(t),X_2(t)):t\geq 0\}$ and the continuity property of $X_2$, we get
	\begin{align}\label{ver1}
	&e^{-\delta t} H(X_1^{\piCi}(t),X_2(t)) \nonumber\\
	=~&H(x_1,x_2) - \int_{0-}^{t}\delta e^{-\delta s}H(X_1^{\piCi}(s-),X_2(s-))ds\nonumber\\
	&+\int_{0-}^{t}e^{-\delta s}\frac{\partial}{\partial x_1}H(X_1^{\piCi}(s-),X_2(s-))dX_1^{{\piCi},c}(s)+\int_{0-}^{t}e^{-\delta s}\frac{\partial}{\partial x_2}H(X_1^{\piCi}(s-),X_2(s-))dX_2(s)\nonumber\\
	&+\int_{0-}^{t}e^{-\delta s} \frac{\sigma_A^2 {(X_1^{\piCi}(s-))}^2}{2} \frac{\partial^2}{\partial {x_1}^2}H(X_1^{\piCi}(s-),X_2(s-))ds+\int_{0-}^{t}e^{-\delta s}\frac{\sigma_L^2 {(X_2(s-))}^2}{2}\frac{\partial^2}{\partial {x_2}^2}H(X_1^{\piCi}(s-),X_2(s-))ds\nonumber\\
	&+\int_{0-}^{t}e^{-\delta s} \rho\sigma_A\sigma_L X_1^{\piCi}(s-)X_2(s-)
	\frac{\partial^2}{\partial x_1 \partial x_2} H(X_1^{\piCi}(s-),X_2(s-))ds\nonumber\\
	&+\sum_{s\leq t} e^{-\delta s}
	\big[H(X^{\piCi}_1(s-)+\Delta X^{\piCi}_1(s),X_2(s)) - H(X^{\piCi}_1(s-),X_2(s-))\big].
	\end{align}
	Rearranging the terms in (\ref{ver1}), using $dX_1^{{\piCi},c}(s)=\mu_A ds+\sigma_A dW(s) - D^{{\piCi},c}(s) + E^{\piCi,c}(s)$ (and similarly for $dX_2(s)$), we obtain
	\begin{align}\begin{split}\label{ver2}
	e^{-\delta t}H(X_1^{\piCi}(t),X_2(t)) 
	=~&H(x_1,x_2) + \int_{0-}^{t}e^{-\delta s}(\mathscr{A}-\delta)H(X_1^{\piCi}(s-),X_2(s-))ds\\
	&+\int_{0-}^{t}e^{-\delta s}\frac{\partial}{\partial x_1}H(X_1^{\piCi}(s-),X_2(s-))\sigma_A X_1^{\piCi}(s-) dW_1(s)\\
	&+\int_{0-}^{t}e^{-\delta s}\frac{\partial}{\partial x_2}H(X_1^{\piCi}(s-),X_2(s-))\rho\sigma_L X_2(s-) dW_1(s)\\
	&+\int_{0-}^{t}e^{-\delta s}\frac{\partial}{\partial x_2}H(X_1^{\piCi}(s-),X_2(s-))\sqrt{1-\rho^2}\sigma_L X_2(s-) dW_2(s)\\
	&+\sum_{s\leq t}
		e^{-\delta s}
	\big[H(X^{\piCi}_1(s-)+\Delta X^{\piCi}_1(s),X_2(s)) - H(X^{\piCi}_1(s-),X_2(s-))\big]\\
	&-\int_{0-}^{t}e^{-\delta s}\frac{\partial}{\partial x_1}H(X_1^{\piCi}(s),X_2(s))dD^{{\piCi},c}(s)\\
	&+\int_{0-}^{t}e^{-\delta s}\frac{\partial}{\partial x_1}H(X_1^{\piCi}(s),X_2(s))dE^{\piCi,c}(s).
	\end{split}\end{align}
	Denote the sum of the second, third and fourth integral as  $M(t)$. After adding and subtracting some terms, (\ref{ver2}) becomes
	\begin{align}\label{ver3}\begin{split}
	&e^{-\delta t}H(X_1^{\piCi}(t),X_2(t))\\
	=~&H(x_1,x_2) + \int_{0-}^{t}e^{-\delta s}(\mathscr{A}-\delta)H(X_1^{\piCi}(s-),X_2(s-))ds\\
	&+M(t)\\
	&+\sum_{s<t,\Delta X^\pi_1(t)>0} e^{-\delta s}
	\big[H(X^{\piCi}_1(s-)+\Delta X^{\piCi}_1(s),X_2(s-)) -\kappa \Delta X_1^{\piCi}(s) - H(X^{\piCi}_1(s-),X_2(s))\big]\\
	&+\sum_{s<t,\Delta X^\pi_1(t)<0} e^{-\delta s}
	\big[H(X^{\piCi}_1(s-)+\Delta X^{\piCi}_1(s),X_2(s-)) - \Delta X^{\piCi}_1(s) - H(X^{\piCi}_1(s-),X_2(s))\big]\\
	&+\int_{0-}^{t}e^{-\delta s} \big(1-\frac{\partial}{\partial x_1}H(X_1^{\piCi}(s),X_2(s))\big)dD^{{\piCi},c}(s)\\
	&+\int_{0-}^{t}e^{-\delta s} \big(\frac{\partial}{\partial x_1}H(X_1^{\piCi}(s),X_2(s))-\kappa\big)dE^{\piCi,c}(s)\\
	&-\Big[\int_{0-}^{t}e^{-\delta s}dD^{{\piCi},c}(s)-\sum_{s<t,\Delta X^\pi_1(t)<0} e^{-\delta s}\Delta X^{\piCi}_1(s)-\sum_{s<t,\Delta X^\pi_1(t)>0} e^{-\delta s}\kappa\Delta X^{\piCi}_1(s)
	-\kappa \int_{0-}^{t}e^{-\delta s}dE^{\piCi,c}(s)\Big].
	\end{split}\end{align}
	Note the assumption that the strategy $\pi$ is efficient implies that $\Delta X^\pi_1(t)<0$ corresponds to a dividend payment at time $t$ while $\Delta X^\pi_1(t)>0$ corresponds to a capital injection, i.e.
	\begin{equation}
	\Delta X^\pi_1(t)=\begin{cases}
	\Delta D^\pi(t), &\text{if } \Delta X^\pi_1(t)<0.\\
	\Delta E^\pi(t), &\text{if } \Delta X^\pi_1(t)>0.
	\end{cases}
	\end{equation}
	Therefore, we can rewrite (\ref{ver3}) as 
	\begin{align}\label{ver4}\begin{split}
	&e^{-\delta t}H(X_1^{\piCi}(t),X_2(t))\\
	=~&H(x_1,x_2) + \int_{0-}^{t}e^{-\delta s}(\mathscr{A}-\delta)H(X_1^{\piCi}(s-),X_2(s-))ds\\
	&+M(t)\\
	&+\sum_{s<t} e^{-\delta s}
	\big[H(X^{\piCi}_1(s-)+\Delta E^{\piCi}(s),X_2(s-)) -\kappa \Delta E^{\piCi}(s) - H(X^{\piCi}_1(s-),X_2(s))\big]\\
	&+\sum_{s<t} e^{-\delta s}
	\big[H(X^{\piCi}_1(s-)-\Delta D^{\piCi}(s),X_2(s-)) + \Delta D^{\piCi}(s) - H(X^{\piCi}_1(s-),X_2(s))\big]\\
	&+\int_{0-}^{t}e^{-\delta s} \big(1-\frac{\partial}{\partial x_1}H(X_1^{\piCi}(s),X_2(s))\big)dD^{{\piCi},c}(s)\\
	&+\int_{0-}^{t}e^{-\delta s} \big(\frac{\partial}{\partial x_1}H(X_1^{\piCi}(s),X_2(s))-\kappa\big)dE^{\piCi,c}(s)\\
	&-\Big[\int_{0-}^{t}e^{-\delta s}dD^{{\piCi},c}(s)+\sum_{s<t} e^{-\delta s}\Delta D^{\piCi}(s)-\sum_{s<t} e^{-\delta s}\kappa\Delta E^{\piCi}(s)
	-\kappa \int_{0-}^{t}e^{-\delta s}dE^{\piCi,c}(s)\Big].
	\end{split}\end{align}
	Using (\ref{vercon1})-(\ref{vercon4}), (\ref{vercon5}) and (\ref{vercon5b}), we can deduce that $\{M(t);t\geq 0\}$ is a zero-mean martingale and 
	\begin{equation*}
	\Big[\int_{0-}^{t}e^{-\delta s}dD^{{\piCi},c}(s)+\sum_{s<t} e^{-\delta s}\Delta D^{\piCi}(s)-\sum_{s<t} e^{-\delta s}\kappa\Delta E^{\piCi}(s)
	-\kappa \int_{0-}^{t}e^{-\delta s}dE^{\piCi,c}(s)\Big]\leq H(x_1,x_2)+M(t).
	\end{equation*}
	Therefore, via taking expectation, we get
		\begin{equation*}
	\Ex\Big[\int_{0-}^{t}e^{-\delta s}dD^{{\piCi}}(s)
	-\kappa \int_{0-}^{t}e^{-\delta s}dE^{\piCi}(s)\Big]\leq H(x_1,x_2).
	\end{equation*}
	Finally, taking $\limsup$ on both sides yields the result.
\end{proof}

\subsection{The barrier strategy with capital injections}\label{S::capitalInjections:B}
  \citet*{GeSh03} match inflow and outflow in the situation where $\kappa=1$. Here, we set $\kappa>1$ and obtain the expected present value of outflow minus inflow if a barrier strategy is applied. The distribution barrier that maximizes this difference exists and is unique. The differential equation characterising the value function in this case has the following form: 
\begin{align}\begin{split}\label{HJB:capital:injection:B}
0=&\max\bigg\{(\mathscr{A}-\delta) \FCi(x_1,x_2),1-\pa{x_1}\FCi(x_1,x_2),\pa{x_1}\FCi(x_1,x_2)-\kappa\bigg\}.
\end{split}\end{align}

We conjecture that an optimal strategy in this formulation is a strategy  $\pi^{\beta,\gamma}$ with two reflecting barriers: a dividend barrier $\beta$ and a capital injection barrier $\gamma$. Of course, for the strategy to make sense and admissible, we require $\alpha_0\leq\gamma<\beta$. {\color{black} Under the reflecting barriers strategy, when the modified funding ratio is above $\beta$, the amount $x_1-\beta x_2$ is paid as dividend, then the modified asset value is reduced to $\beta x_2$; when the modified funding ratio is below $\gamma$, the cost of capital injection is $\kappa (x_1-\gamma x_2)$, then the modified asset value is increased to $\gamma x_2$.} Therefore, the expected present value of dividends for a doubly-reflected barrier with barriers $(\beta,\gamma)$, denoted by $G^{\beta,\gamma}$, is given by
\begin{align}\begin{split}\label{barrier:double:B}
G^{\beta,\gamma}(x_1,x_2)=\left\{\begin{tabular}{ll}
$G(x_1,x_2;\beta,\gamma),$&$ (x_1,x_2)\in L([\gamma,\beta])$,\\
$(x_1-\beta x_2)+G(\beta x_2,x_2;\beta,\gamma),$&$(x_1,x_2)\in L((\beta,\infty))$,\\
$\kappa(x_1-\gamma x_2)+G(\gamma x_2,x_2;\beta,\gamma),$&$ (x_1,x_2)\in L([\alpha_0,\gamma])$,
\end{tabular}
\right. 
\end{split}\end{align}
where $G(x_1,x_2;\beta,\gamma)$ is given by the differential equation
\begin{align*}
(\mathscr{A}-\delta) G(x_1,x_2;\beta,\gamma)=0 \text{ for }{\gamma}\leq \frac{x_1}{x_2}\leq \beta.
\end{align*}
Since we have forced capital injections we do not have a boundary condition at the ruin level as in (\ref{BC:D0:B}). Instead, we have a boundary condition at the capital injection level which can be obtained in the same way as (\ref{BC:D1:B}). In total, the boundary conditions are
\begin{eqnarray}
\pa{x_1}G\left(x_1,x_2;\beta,\gamma\right)|_{x_1=\beta x_2-}&=&1,\label{BC:D1b:B}\\
\pa{x_1}G\left(x_1,x_2;\beta,\gamma\right)|_{x_1=\gamma x_2+}&=&\kappa,\label{BC:D0b:B}
\end{eqnarray}
For a heuristic derivation of conditions (\ref{BC:D1b:B}) and (\ref{BC:D0b:B}), see \citet*[Section 3.1]{AvShWo11}. {\color{black} Analogous to the solution for $(\mathscr{A}-\delta) G(x_1,x_2;\beta)=0$ as given in (\ref{eqn:general:representation:value:function}), the solution of $(\mathscr{A}-\delta) G(x_1,x_2;\beta,\gamma)=0$ is in the form
$$G(x_1,x_2;\beta,\gamma)=C_1{x_1}^{\zeta_1}{x_2}^{1-\zeta_1}+C_2{x_1}^{\zeta_2}{x_2}^{1-\zeta_2},$$
where $C_1$ and $C_2$ are two constants need to be solved by using conditions (\ref{BC:D1b:B}) and (\ref{BC:D0b:B}).} Using condition (\ref{BC:D0b:B}), we obtain that $C_1$ and $C_2$ fulfill the equation
\begin{align}\begin{split}
C_1\zeta_1\left(\gamma x_2\right)^{\zeta_1-1}x_2^{1-\zeta_1}+C_2\zeta_2\left(\gamma x_2\right)^{\zeta_2-1}x_2^{1-\zeta_2}&=\kappa.
\end{split}\end{align}
This means that 
\begin{align*}
C_2&=\frac{\kappa-C_1\zeta_1\gamma^{\zeta_1-1}}{\zeta_2\gamma^{\zeta_2-1}}=\frac{\kappa\gamma^{1-\zeta_2}-C_1\zeta_1\gamma^{\zeta_1-\zeta_2}}{\zeta_2}.
\end{align*}
Further using condition (\ref{BC:D1b:B}), we obtain that
\begin{align*}
C_1\zeta_1\left(x_2\beta\right)^{\zeta_1-1}x_2^{1-\zeta_1}+\frac{\kappa-C_1\zeta_1\gamma^{\zeta_1-1}}{\zeta_2\gamma^{\zeta_2-1}}\zeta_2\left(x_2\beta\right)^{\zeta_2-1}x_2^{1-\zeta_2}&=1.
\end{align*}
This implies that $C_1$ is given by
\begin{align}\begin{split}\label{C1:capitaL:injection:B}
C_1&=\frac {{\gamma}^{\zeta_2-1}-\kappa\beta^{\zeta_2-1}}{\zeta_1 \left( {\beta}^{\zeta_1-1}\gamma^{\zeta_2-1}-{\gamma}^{\zeta_1-1}\beta^{\zeta_2-1} \right)}=\frac {1-\kappa{\beta}^{\zeta_2-1}{\gamma}^{1-\zeta_2}}{\zeta_1 \left( {\beta}^{\zeta_1-1}-{\gamma}^{\zeta_1-\zeta_2}{\beta}^{\zeta_2-1} \right)}.
\end{split}\end{align}
That is, in total we get that the value function for a strategy with dividend barrier $\beta$ and capital injection barrier $\gamma$ with $\beta >\gamma$ is 
\begin{align}\begin{split}\label{value:fct:capital:explixcit:B}
G(x_1,x_2;\beta,\gamma)=C_1 x_1^{\zeta_1}x_2^{1-\zeta_1}+\frac{\kappa\gamma^{1-\zeta_2}-C_1\zeta_1\gamma^{\zeta_1-\zeta_2}}{\zeta_2}x_1^{\zeta_2}x_2^{1-\zeta_2},
\end{split}\end{align}
where $C_1$ is given by (\ref{C1:capitaL:injection:B}).

\subsection{The optimal barriers}

In this section, we propose the choice of the optimal barriers, denoted as $(\beta_2^*,\gamma^*)$. Because of transaction costs it makes sense that the optimal rescue level is the lowest possible (i.e. $\gamma^*=\alpha_0$), which is our choice of $\gamma^*$. 

\begin{remark}
	Note that this may not hold any more if transaction costs are not constant. 
	Furthermore, we conjecture that the form of the optimal distribution strategy is as before, that is, of the barrier type. 
\end{remark}

  For $\beta_2^*$, it makes sense to maximize (\ref{value:fct:capital:explixcit:B}) with respect to $\beta$, for $\frac{x_1}{x_2}\in[\gamma^*,\beta]$. By rearranging the terms, we see that
  \begin{align*}
  G(x_1,x_2;\beta,\gamma)=\frac{\kappa\gamma^{1-\zeta_2}x_1^{\zeta_2}x_2^{1-\zeta_2}}{\zeta_2}+C_1 x_1^{\zeta_1}x_2^{1-\zeta_1}+C_1\frac{-\zeta_1\gamma^{\zeta_1-\zeta_2}}{\zeta_2}x_1^{\zeta_2}x_2^{1-\zeta_2},
  \end{align*} 
  where the first term of the right-hand side does not depend on $\beta$ and $x_1^{\zeta_1}x_2^{1-\zeta_1}$, $x_1^{\zeta_2}x_2^{1-\zeta_2}$ and $\frac{-\zeta_1\gamma^{\zeta_1-\zeta_2}}{\zeta_2}$ are all positive terms. That is, to maximize (\ref{value:fct:capital:explixcit:B}) we need to maximize $C_1$ as a function of $\beta$.

 The partial derivative of $C_1$ wrt.\ $\beta$ is given by
\begin{align*}
\pa{\beta}C_1=&\pa{\beta}\left(\frac {1-\kappa{\beta}^{\zeta_2-1}{\gamma}^{1-\zeta_2}}{\zeta_1 \left( {\beta}^{\zeta_1-1}-{\gamma}^{\zeta_1-\zeta_2}{\beta}^{\zeta_2-1} \right)}\right)\\
=&\frac{1}{\zeta_1^2 \left( {\beta}^{\zeta_1-1}-{\gamma}^{\zeta_1-\zeta_2}{\beta}^{\zeta_2-1} \right)^2}
\Big(-(\zeta_2-1)\kappa{\beta}^{\zeta_2-2}{\gamma}^{1-\zeta_2}\zeta_1 \left( {\beta}^{\zeta_1-1}-{\gamma}^{\zeta_1-\zeta_2}{\beta}^{\zeta_2-1} \right)\\
&-\left(1-\kappa{\beta}^{\zeta_2-1}{\gamma}^{1-\zeta_2}\right)\zeta_1 \left( (\zeta_1-1){\beta}^{\zeta_1-2}-(\zeta_2-1){\gamma}^{\zeta_1-\zeta_2}{\beta}^{\zeta_2-2} \right)\Big).
\end{align*}
Setting this term equal to zero leads to the equation

\begin{align}\begin{split}
\label{capital:injection:eqn:for:beta:B:former}
&-(\zeta_2-1)\kappa{\beta}^{\zeta_2-2}{\gamma}^{1-\zeta_2}\zeta_1 \left( {\beta}^{\zeta_1-1}-{\gamma}^{\zeta_1-\zeta_2}{\beta}^{\zeta_2-1} \right)\\
=&\left(1-\kappa{\beta}^{\zeta_2-1}{\gamma}^{1-\zeta_2}\right)\zeta_1 \left( (\zeta_1-1){\beta}^{\zeta_1-2}-(\zeta_2-1){\gamma}^{\zeta_1-\zeta_2}{\beta}^{\zeta_2-2} \right).
\end{split}\end{align}	
We can rewrite this representation for the optimal upper barrier 
and obtain the following representation for $\kappa$: 
\begin{align}\begin{split}\label{kappa:eqn:for:case:B}
\kappa
=\frac{(\zeta_1-1)-(\zeta_2-1)\left(\frac{\beta}{\gamma}\right)^{\zeta_2-\zeta_1}}{(\zeta_1-\zeta_2)\left(\frac{\beta}{\gamma}\right)^{\zeta_2-1}}.
\end{split}\end{align}
Equation (\ref{capital:injection:eqn:for:beta:B:former}) gives us an equation of the form (dividing by $\beta^{\zeta_2-3}$) 
\begin{align}\begin{split}\label{capital:injection:eqn:for:beta:B}
\zeta_{1} \left( {\gamma}^{1-\zeta_{2}}\kappa(\zeta_{1}-\zeta_{2}){\beta}^{\zeta_{1}}+(\zeta_2-1){\gamma}^{\zeta_{1}-\zeta_{2}}\beta+(1-\zeta_1){\beta}^{1-\zeta_{2}+\zeta_{1}} \right)=0.
\end{split}\end{align}
The ``optimal'' dividend barrier level $\beta_2^*$ is defined as the solution to (\ref{capital:injection:eqn:for:beta:B}).
The existence and uniqueness of the optimal barrier level is given by the following lemma.
\begin{lemma}\label{lemma:CI:existence:barrier:B}
The ``optimal'' barrier level $\beta_2^*$ (defined as the solution to (\ref{capital:injection:eqn:for:beta:B})) exists and is unique. 
\end{lemma}
\begin{proof}[Proof of Lemma \ref{lemma:CI:existence:barrier:B}]
  The problem is to show that the equation (\ref{capital:injection:eqn:for:beta:B}) has a unique root. We denote by $\Psi$ the function 
\begin{align*}
\Psi(\beta)&=\zeta_{1} \left( {\beta}^{\zeta_{1}}{\gamma}^{1-\zeta_{2}}\kappa\zeta_{1}-{\beta}^{\zeta_{1}}{\gamma}^{1-\zeta_{2}}\kappa\zeta_{2}+{\gamma}^{\zeta_{1}-\zeta_{2}}\beta\zeta_{2}-{\beta}^{1-\zeta_{2}+\zeta_{1}}\zeta_{1}-{\gamma}^{\zeta_{1}-\zeta_{2}}\beta+{\beta}^{1-\zeta_{2}+\zeta_{1}} \right).
\end{align*}
We want to prove that $\sgn\left(\Psi(\gamma)\right)\neq\sgn\left(\lim_{\beta\to\infty}\Psi(\beta)\right)$, and that 
\begin{align*}
\sgn\left(\Psi'(x)\right)=\sgn\left(\Psi'(y)\right)
\end{align*}
for $x,y>\gamma$.
\begin{enumerate}
 \item 
\begin{align*}
\Psi(\gamma)=&\zeta_{1}{\gamma}^{1-\zeta_{2}+\zeta_{1}} \left( \kappa\zeta_{1}-\kappa\zeta_{2}-\zeta_{1}+\zeta_{2} \right) 
>0.
\end{align*}
\item 
\begin{align*}
\lim_{\beta\to\infty}\Psi(\beta)&=\lim_{\beta\to\infty}\zeta_{1} \left({\gamma}^{\zeta_{1}-\zeta_{2}}\beta\zeta_{2}-{\gamma}^{\zeta_{1}-\zeta_{2}}\beta\right)=-\infty. 
\end{align*}
\item For $\beta\geq \gamma$,
\begin{align*}
\Psi'(\beta)=&-\zeta_{1} \bigg( -{\beta}^{\zeta_{1}-1}{\zeta_{1}}^{2}{\gamma}^{1-\zeta_{2}}\kappa+{\beta}^{\zeta_{1}-1}\zeta_{1}{\gamma}^{1-\zeta_{2}}\kappa\zeta_{2}+{\beta}^{\zeta_{1}-\zeta_{2}}{\zeta_{1}}^{2}-{\beta}^{\zeta_{1}-\zeta_{2}}\zeta_{1}\zeta_{2}\\
&\phantom{-\zeta_{1} \bigg(}+{\beta}^{\zeta_{1}-\zeta_{2}}\zeta_{2}-{\gamma}^{\zeta_{1}-\zeta_{2}}\zeta_{2}-{\beta}^{\zeta_{1}-\zeta_{2}}+{\gamma}^{\zeta_{1}-\zeta_{2}} \bigg).
\end{align*} 
We want to show that the inner part of the delimiters is negative. It is given by
\begin{align}\begin{split}\label{derivative:eqn:B}
&\kappa{\beta}^{\zeta_{1}-1}{\gamma}^{1-\zeta_{2}}\left(-{\zeta_{1}}^{2}+\zeta_{1}\zeta_{2}\right)+{\beta}^{\zeta_{1}-\zeta_{2}}\left({\zeta_{1}}^{2}-\zeta_{1}\zeta_{2}+\zeta_{2}-1\right)-{\gamma}^{\zeta_{1}-\zeta_{2}}\left(\zeta_{2}-1\right).
\end{split}\end{align}
To get that (\ref{derivative:eqn:B}) is negative, we need the two following conditions to be fulfilled:
\begin{align}\begin{split}
{\beta}^{\zeta_{1}-1}{\gamma}^{1-\zeta_{2}}\geq {\beta}^{\zeta_{1}-\zeta_{2}} \text{ and } {\gamma}^{\zeta_{1}-\zeta_{2}}\geq {\beta}^{\zeta_{1}-\zeta_{2}}. 
\end{split}\end{align}
Since $\left(\frac{\gamma}{\beta}\right)^{1-\zeta_2}\geq 1$ and $\left(\frac{\gamma}{\beta}\right)^{\zeta_1-\zeta_2}\geq 1$, both inequalities holds and $\Psi'(\beta)<0$. 
\end{enumerate}
That is, the optimal barrier exists and is unique.
\end{proof}

Having chosen the optimal barriers $(\beta^*_2,\gamma^*)$, we are left to prove the optimality of the barrier strategy.

\subsection{Verification of all the conditions of verification lemma}
  It remains to show that our candidate strategy $\pi^{\beta^*_2,\alpha_0}$ satisfies all the conditions purposed in the lemma. We start by showing concavity of the value function.
\begin{lemma}\label{thm:ci:2:order:negative}It holds that
\begin{align*}\begin{split}
\paTo{x_1}G(x_1,x_2;\beta^*_2,\gamma^*)\leq 0.
\end{split}\end{align*}
\end{lemma}
\begin{proof}
We use the representation (\ref{value:fct:capital:explixcit:B}) for the optimal levels $\gamma^*$ and $\beta^*_2$:
\begin{eqnarray}\label{thmp1}
&&\paTo{x_1}G(x_1,x_2;\beta^*_2,\gamma^*)\nonumber\\
&=&\paTo{x_1}\left(C_1 x_1^{\zeta_1}x_2^{1-\zeta_1}+\frac{\kappa\left(\gamma^*\right)^{1-\zeta_2}-C_1\zeta_1\left(\gamma^*\right)^{\zeta_1-\zeta_2}}{\zeta_2}x_1^{\zeta_2}x_2^{1-\zeta_2}\right) \nonumber\\
&=&\zeta_1(\zeta_1-1)C_1 x_1^{\zeta_1-2}x_2^{1-\zeta_1}+\zeta_2(\zeta_2-1)\frac{\kappa\left(\gamma^*\right)^{1-\zeta_2}-C_1\zeta_1\left(\gamma^*\right)^{\zeta_1-\zeta_2}}{\zeta_2}x_1^{\zeta_2-2}x_2^{1-\zeta_2}\nonumber\\
&=&x_1^{-2}x_2 \left\lbrace \zeta_1(\zeta_1-1)C_1 {\left(\frac{x_1}{x_2}\right)}^{\zeta_1}
+(\zeta_2-1)\left(\kappa\left(\gamma^*\right)^{1-\zeta_2}-C_1\zeta_1\left(\gamma^*\right)^{\zeta_1-\zeta_2}\right){\left(\frac{x_1}{x_2}\right)}^{\zeta_1}\right\rbrace 
\end{eqnarray}
Note we have that
\begin{align*}\begin{split}
C_1=\frac {1-\kappa{\beta^*_2}^{\zeta_2-1}{\gamma^*}^{1-\zeta_2}}{\zeta_1 \left( {\beta^*_2}^{\zeta_1-1}-{\gamma^*}^{\zeta_1-\zeta_2}{\beta^*_2}^{\zeta_2-1} \right)}=\frac {1-\kappa\left(\frac{{\beta^*_2}}{\gamma^*}\right)^{\zeta_2-1}}{\zeta_1 \left( {\beta^*_2}^{\zeta_1-1}-{\gamma^*}^{\zeta_1-\zeta_2}{\beta^*_2}^{\zeta_2-1} \right)}, 
\end{split}\end{align*}
where the denominator is positive. By substituting the expression of $C_1$ to (\ref{thmp1}) and multiplying the whole expression by $x_1^2x_2^{-1}\left( {\beta^*_2}^{\zeta_1-1}-{\gamma^*}^{\zeta_1-\zeta_2}{\beta^*_2}^{\zeta_2-1} \right)$ (negative by using $\zeta_1<0$ and $\zeta_2>1$ which are discussed below (\ref{sol:quad:eqn})), we have 
\begin{align}\label{thmp2}
&\paTo{x_1}G(x_1,x_2;\beta^*_2,\gamma^*)\leq 0\nonumber\\
\iff& \left(1-\kappa{\left(\frac{\beta^*_2}{\gamma^*}\right)}^{\zeta_2-1}\right)(\zeta_1-1){\left(\frac{x_1}{x_2}\right)}^{\zeta_1}
\nonumber\\
&+\left[ \kappa {\gamma^*}^{1-\zeta_2}\left( {\beta^*_2}^{\zeta_1-1}-{\gamma^*}^{\zeta_1-\zeta_2}{\beta^*_2}^{\zeta_2-1} \right) - {\gamma^*}^{\zeta_1-\zeta_2}\left(1-\kappa{\left(\frac{\beta^*_2}{\gamma^*}\right)}^{\zeta_2-1}\right)\right] (\zeta_2-1){\left(\frac{x_1}{x_2}\right)}^{\zeta_2}\geq 0\nonumber\\
\iff& \left(1-\kappa{\left(\frac{\beta^*_2}{\gamma^*}\right)}^{\zeta_2-1}\right)(\zeta_1-1)
+\left[ \kappa {\gamma^*}^{1-\zeta_2}\left( {\gamma^*}^{\zeta_2-\zeta_1}{\beta^*_2}^{\zeta_1-1}-{\beta^*_2}^{\zeta_2-1} \right) - \left(1-\kappa{\left(\frac{\beta^*_2}{\gamma^*}\right)}^{\zeta_2-1}\right)\right] \nonumber\\
&\times (\zeta_2-1){\left(\frac{x_1}{\gamma^* x_2}\right)}^{\zeta_2-\zeta_1}\geq 0\nonumber\\
\iff& \left(1-\kappa{\left(\frac{\beta^*_2}{\gamma^*}\right)}^{\zeta_2-1}\right)(\zeta_1-1)\nonumber\\
\phantom{\iff}&+\left[ \kappa {\gamma^*}^{1-\zeta_1}{\beta^*_2}^{\zeta_1-1} -\kappa{\left(\frac{\beta^*_2}{\gamma^*}\right)}^{\zeta_2-1} - 1+\kappa{\left(\frac{\beta^*_2}{\gamma^*}\right)}^{\zeta_2-1}\right](\zeta_2-1){\left(\frac{x_1}{\gamma^* x_2}\right)}^{\zeta_2-\zeta_1}\geq 0\nonumber\\
\iff&
\left(1-\kappa{\left(\frac{\beta^*_2}{\gamma^*}\right)}^{\zeta_2-1}\right)(\zeta_1-1)
+\left[\kappa{\left(\frac{\beta^*_2}{\gamma^*}\right)}^{\zeta_1-1} - 1\right](\zeta_2-1){\left(\frac{x_1}{\gamma^* x_2}\right)}^{\zeta_2-\zeta_1}\geq 0.
\end{align}
Now, we can use $\kappa=\frac{(\zeta_1-1)-(\zeta_2-1){\left(\frac{\beta^*_2}{\gamma^*}\right)}^{\zeta_2-\zeta_1}}{(\zeta_1-\zeta_2){\left(\frac{\beta^*_2}{\gamma^*}\right)}^{\zeta_2-1}}$ in (\ref{thmp2}) to simplify $1-\kappa{\left(\frac{\beta^*_2}{\gamma^*}\right)}^{\zeta_2-1}$ and $\kappa{\left(\frac{\beta^*_2}{\gamma^*}\right)}^{\zeta_1-1}-1$. We have
\begin{align*}
&1-\kappa{\left(\frac{\beta^*_2}{\gamma^*}\right)}^{\zeta_2-1}\\
=~&\frac{1}{(\zeta_1-\zeta_2){\left(\frac{\beta^*_2}{\gamma^*}\right)}^{\zeta_2-1}}
\left((\zeta_1-\zeta_2){\left(\frac{\beta^*_2}{\gamma^*}\right)}^{\zeta_2-1}-\left((\zeta_1-1)-(\zeta_2-1){\left(\frac{\beta^*_2}{\gamma^*}\right)}^{\zeta_2-\zeta_1}\right){\left(\frac{\beta^*_2}{\gamma^*}\right)}^{\zeta_2-1} \right) \\
=~&\frac{1}{\zeta_1-\zeta_2}\left((\zeta_1-\zeta_2)-(\zeta_1-1)+(\zeta_2-1) {\left(\frac{\beta^*_2}{\gamma^*}\right)}^{\zeta_2-\zeta_1}\right)\\
=~&\frac{1}{\zeta_1-\zeta_2}\left((1-\zeta_2)+(\zeta_2-1) {\left(\frac{\beta^*_2}{\gamma^*}\right)}^{\zeta_2-\zeta_1}\right)\\
=~&\frac{1-\zeta_2}{\zeta_1-\zeta_2}\left(1- {\left(\frac{\beta^*_2}{\gamma^*}\right)}^{\zeta_2-\zeta_1}\right)
\end{align*}
and
\begin{align*}
&\kappa{\left(\frac{\beta^*_2}{\gamma^*}\right)}^{\zeta_1-1}-1\\
=~&\frac{1}{(\zeta_1-\zeta_2){\left(\frac{\beta^*_2}{\gamma^*}\right)}^{\zeta_2-1}}
\left( \left((\zeta_1-1)-(\zeta_2-1){\left(\frac{\beta^*_2}{\gamma^*}\right)}^{\zeta_2-\zeta_1}\right){\left(\frac{\beta^*_2}{\gamma^*}\right)}^{\zeta_1-1} - (\zeta_1-\zeta_2){\left(\frac{\beta^*_2}{\gamma^*}\right)}^{\zeta_2-1} \right) \\
=~&\frac{1}{\zeta_1-\zeta_2}
\left(\big((\zeta_1-1)-(\zeta_2-1){\left(\frac{\beta^*_2}{\gamma^*}\right)}^{\zeta_2-\zeta_1}\big){\left(\frac{\beta^*_2}{\gamma^*}\right)}^{\zeta_1-\zeta_2}- (\zeta_1-\zeta_2)\right)\\
=~&\frac{1}{\zeta_1-\zeta_2}\left((\zeta_1-1){\left(\frac{\beta^*_2}{\gamma^*}\right)}^{\zeta_1-\zeta_2}-\zeta_2+1-\zeta_1+\zeta_2 \right) \\
=~&\frac{\zeta_1-1}{\zeta_1-\zeta_2}\left( {\left(\frac{\beta^*_2}{\gamma^*}\right)}^{\zeta_1-\zeta_2}-1\right).
\end{align*}
Putting these two expressions back to (\ref{thmp2}), we have
\begin{align}
&\paTo{x_1}G(x_1,x_2;\beta^*_2,\gamma^*)\leq 0\nonumber\\
\iff&\frac{1-\zeta_2}{\zeta_1-\zeta_2}\left(1- {\left(\frac{\beta^*_2}{\gamma^*}\right)}^{\zeta_2-\zeta_1}\right)(\zeta_1-1)
+\frac{\zeta_1-1}{\zeta_1-\zeta_2}\left( {\left(\frac{\beta^*_2}{\gamma^*}\right)}^{\zeta_1-\zeta_2}-1\right)(\zeta_2-1){\left(\frac{x_1}{\gamma^* x_2}\right)}^{\zeta_2-\zeta_1}\geq 0\nonumber\\
\iff&
\frac{(1-\zeta_2)(\zeta_1-1)}{\zeta_1-\zeta_2}\left(\big(1- {\left(\frac{\beta^*_2}{\gamma^*}\right)}^{\zeta_2-\zeta_1}\big)
-\big( {\left(\frac{\beta^*_2}{\gamma^*}\right)}^{\zeta_1-\zeta_2}-1\big){\left(\frac{x_1}{\gamma^* x_2}\right)}^{\zeta_2-\zeta_1}\right)\geq 0\nonumber\\
\iff&(\big(1- {\left(\frac{\beta^*_2}{\gamma^*}\right)}^{\zeta_2-\zeta_1}\big)
-\big( {\left(\frac{\beta^*_2}{\gamma^*}\right)}^{\zeta_1-\zeta_2}-1\big){\left(\frac{x_1}{\gamma^* x_2}\right)}^{\zeta_2-\zeta_1}\leq 0 \label{d2geq0}\\
\Longleftarrow &1- {\left(\frac{\beta^*_2}{\gamma^*}\right)}^{\zeta_2-\zeta_1}\leq {\left(\frac{\beta^*_2}{\gamma^*}\right)}^{\zeta_1-\zeta_2}-1 \nonumber\\
\iff& 1\leq \frac{1}{2}\Bigg( {\left(\frac{\beta^*_2}{\gamma^*}\right)}^{\zeta_2-\zeta_1}+{\left(\frac{\beta^*_2}{\gamma^*}\right)}^{\zeta_1-\zeta_2}\Bigg),
\end{align}
where we used the fact that $\frac{x_1}{\gamma^* x_2}\geq 1$ and the last inequality is always true since the right-hand side is an arithmetic sum (of two positive numbers) while the left-hand side is the geometric sum of them.
\end{proof}

\begin{remark}
When $\frac{x_1}{x_2}=\beta^*_2$, the inequality in (\ref{d2geq0}) becomes equality and therefore we have
\begin{align}\begin{split}
\paTo{x_1}G(x_1,x_2;\beta^*_2,\gamma^*)|_{\frac{x_1}{x_2}\nearrow\beta^*_2}=0.
\end{split}\end{align}
\end{remark}

\begin{lemma} \label{lemma:CI:satisfies:requirements:VL}
The value function of a barrier strategy with barrier levels $\gamma^*=\alpha_0$ and $\beta^*_2$ (defined as the solution to (\ref{capital:injection:eqn:for:beta:B})) satisfies all the conditions in the verification Lemma \ref{verlemma3}.
\end{lemma}	
\begin{proof}
For presentation purpose, we denote $G(x_1,x_2)$ for $G(x_1,x_2;\beta^*_2,\gamma^*)$. 
First, we need to show that $G^{\beta^*_2,\gamma^*}(x_1,x_2)\in{\mathscr{C}}^2$. \\
Since $\gamma^*=\alpha_0$, we have only 2 branches to work with, namely $\frac{x_1}{x_2}\in[\alpha_0,\beta^*_2)$ and $\frac{x_1}{x_2}\in[\beta^*_2,\infty)$.
From Lemma \ref{thm:ci:2:order:negative}, we have 
\begin{equation}\label{partiald1}
\frac{\partial{\color{black}^2}}{\partial x_1^2}G(x_1,x_2)|_{\frac{x_1}{x_2}\nearrow\beta^*_2}=0.
\end{equation}
By taking partial derivatives of $G$ using (\ref{value:fct:capital:explixcit:B}), we are able to show that 
\begin{equation}\label{partiald2}
\frac{\partial{\color{black}^2}}{\partial x_2^2}G(x_1,x_2)|_{\frac{x_1}{x_2}\nearrow\beta^*_2}=\frac{\partial{\color{black}^2}}{\partial x_1^2}G(x_1,x_2)|_{\frac{x_1}{x_2}\nearrow\beta^*_2}\times {\beta^*_2}^2 =0
\end{equation}
 and 
 \begin{equation}\label{partiald3}
 \frac{\partial{\color{black}^2}}{\partial x_1 \partial x_2}G(x_1,x_2)|_{\frac{x_1}{x_2}\nearrow\beta^*_2}=\frac{\partial{\color{black}^2}}{\partial x_1^2}G(x_1,x_2)|_{\frac{x_1}{x_2}\nearrow\beta^*_2}\times \beta^*_2=0.
 \end{equation}
Equations (\ref{partiald1}), (\ref{partiald2}) and (\ref{partiald3}) prove that the second-order partial derivatives are continuous. Since the function $G^{\beta^*_2,\gamma^*}$ is first-order differentiable, we can conclude that $G^{\beta^*_2,\gamma^*}$ is second-order differentiable, i.e. $G^{\beta^*_2,\gamma^*}\in{\mathscr{C}}^2$.\\

Next, we show that $(\mathscr{A}-\delta)G^{\beta^*_2,\gamma^*}(x_1,x_2)\leq 0$. We need to show that $(\mathscr{A}-\delta)G^{\beta^*_2,\gamma^*}(x_1,x_2)\leq 0$ holds for the 2 branches $\frac{x_1}{x_2}\in[\alpha_0,\beta^*_2)$ and $\frac{x_1}{x_2}\in[\beta^*_2,\infty)$.
For the first branch, by the form of the value function, we have $(\mathscr{A}-\delta)G^{\beta^*_2,\gamma^*}(x_1,x_2)= 0$. For the second branch, since $G^{\beta^*_2,\gamma^*}\in{\mathscr{C}}^2$ and $G^{\beta^*_2,\gamma^*}(x_1,x_2) = x_1-\beta^*_2 x_2+G(\beta^*_2x_2,x_2)$, we have
\begin{align*}
(\mathscr{A}-\delta)G^{\beta^*_2,\gamma^*}(x_1,x_2)=
~&-\delta(x_1-\beta^*_2 x_2)+(\mathscr{A}-\delta)G(\beta^*_2x_2,x_2)\\
=~&-\delta(x_1-\beta^*_2 x_2)\\
\leq~&0.
\end{align*}
The condition $1\leq\frac{\partial}{\partial{x_1}}G^{\beta^*_2,\gamma^*}(x_1,x_2)\leq\kappa$ is an automatic result from (\ref{BC:D1b:B}), (\ref{BC:D0b:B}) and Lemma \ref{thm:ci:2:order:negative}.\\
Lastly, we show $|\frac{\partial}{\partial{x_2}}G^{\beta^*_2,\gamma^*}(x_1,x_2)|<K$ for some positive constant $K$.\\
For the first branch $\frac{x_1}{x_2}\in[\alpha_0,\beta^*_2)$, we have
\begin{align*}
\left|\frac{\partial}{\partial{x_2}}G^{\beta^*_2,\gamma^*}(x_1,x_2)\right|
=&\left| K_1 (1-\zeta_1){\left( \frac{x_1}{x_2}\right) }^{\zeta_1}+K_2(1-\zeta_2){\left( \frac{x_1}{x_2}\right) }^{\zeta_2}\right| \\
\leq&\left| K_1(1-\zeta_1)\right| {\left( \frac{x_1}{x_2}\right) }^{\zeta_1}+\left| K_2(1-\zeta_2)\right| {\left( \frac{x_1}{x_2}\right) }^{\zeta_2}\\
\leq&\left| K_1(1-\zeta_1)\right| {(\alpha_0) }^{\zeta_1}+\left| K_2(1-\zeta_2)\right| {(\beta^*_2)}^{\zeta_2},
\end{align*}
which is bounded. For the second branch, we have
\begin{align*}
\left|\frac{\partial}{\partial{x_2}}G^{\beta^*_2,\gamma^*}(x_1,x_2)\right|=&|-\beta^*_2|,
\end{align*}
which is also bounded. Take $K=\max\left\{\beta^*_2,K_1(1-\zeta_1)| {(\alpha_0) }^{\zeta_1}+| K_2(1-\zeta_2)| {(\beta^*_2)}^{\zeta_2}\right\}+1$, and by noticing that $G$ is continuously differentiable, we arrive at $\left|\frac{\partial}{\partial{x_2}}G^{\beta^*_2,\gamma^*}(x_1,x_2)\right|<K$ for some positive constant $K$.
\end{proof}

  Since the value function $G^{\beta^*_2,\gamma^*}$ satisfies all the conditions proposed in the verification Lemma \ref{verlemma3}, we conclude that the barrier strategy with level $\gamma=a_0$ and $\beta=\beta^*_2$ is optimal. This is restated in the following.
  
\begin{theorem}
	Define $\gamma^*=\alpha_0$ and $\beta^*_2$ as the solution to (\ref{capital:injection:eqn:for:beta:B}), the doubly-reflected barrier strategy with barriers $(\beta^*_2,\gamma^*)$ is optimal.
\end{theorem}

\section{Numerical illustrations} \label{S_NumIll}

\begin{figure}[htb]
\begin{center}
\includegraphics[width=0.49\textwidth, clip=true, trim = 8mm 0mm 10mm 20mm]{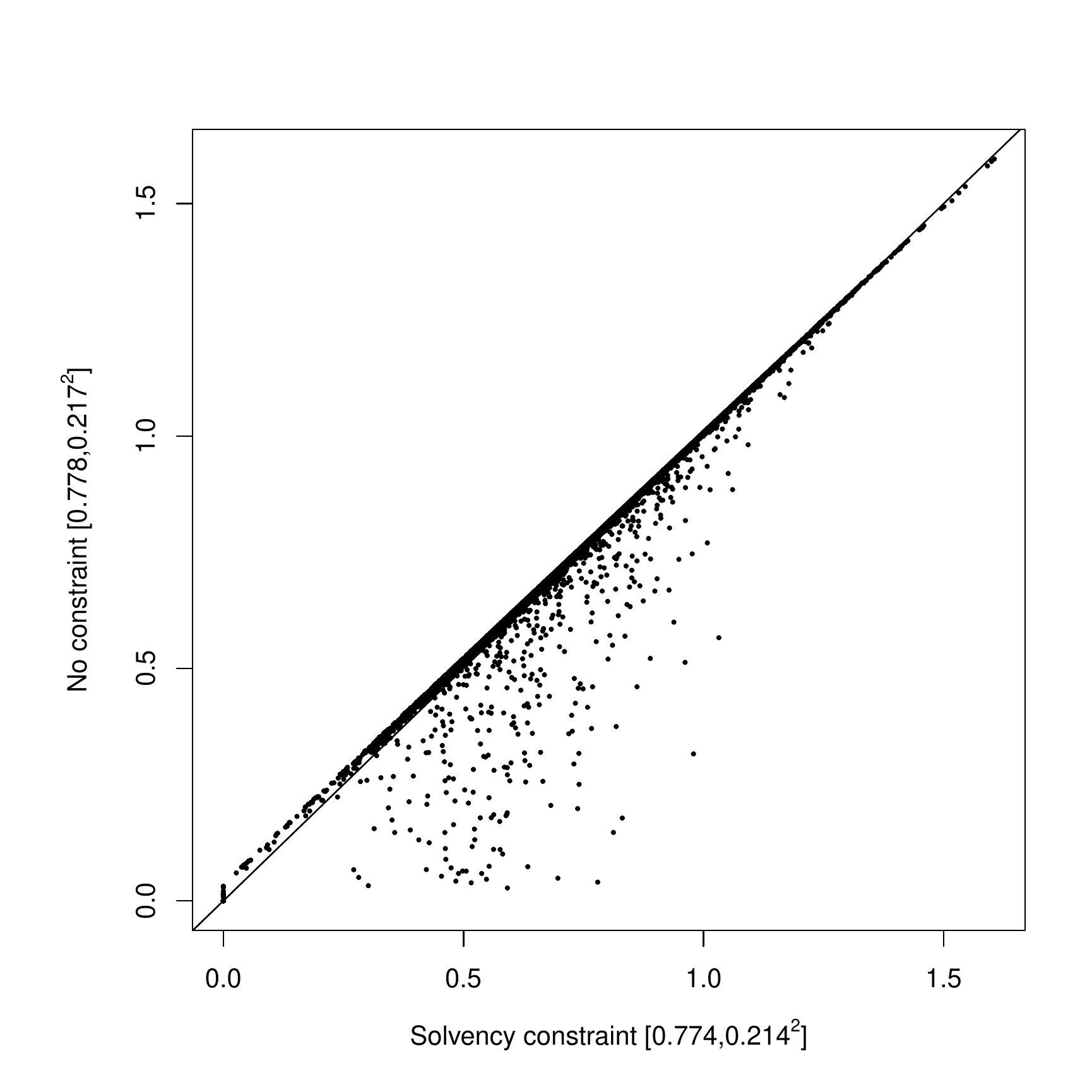}
\includegraphics[width=0.49\textwidth, clip=true, trim = 8mm 0mm 10mm 20mm]{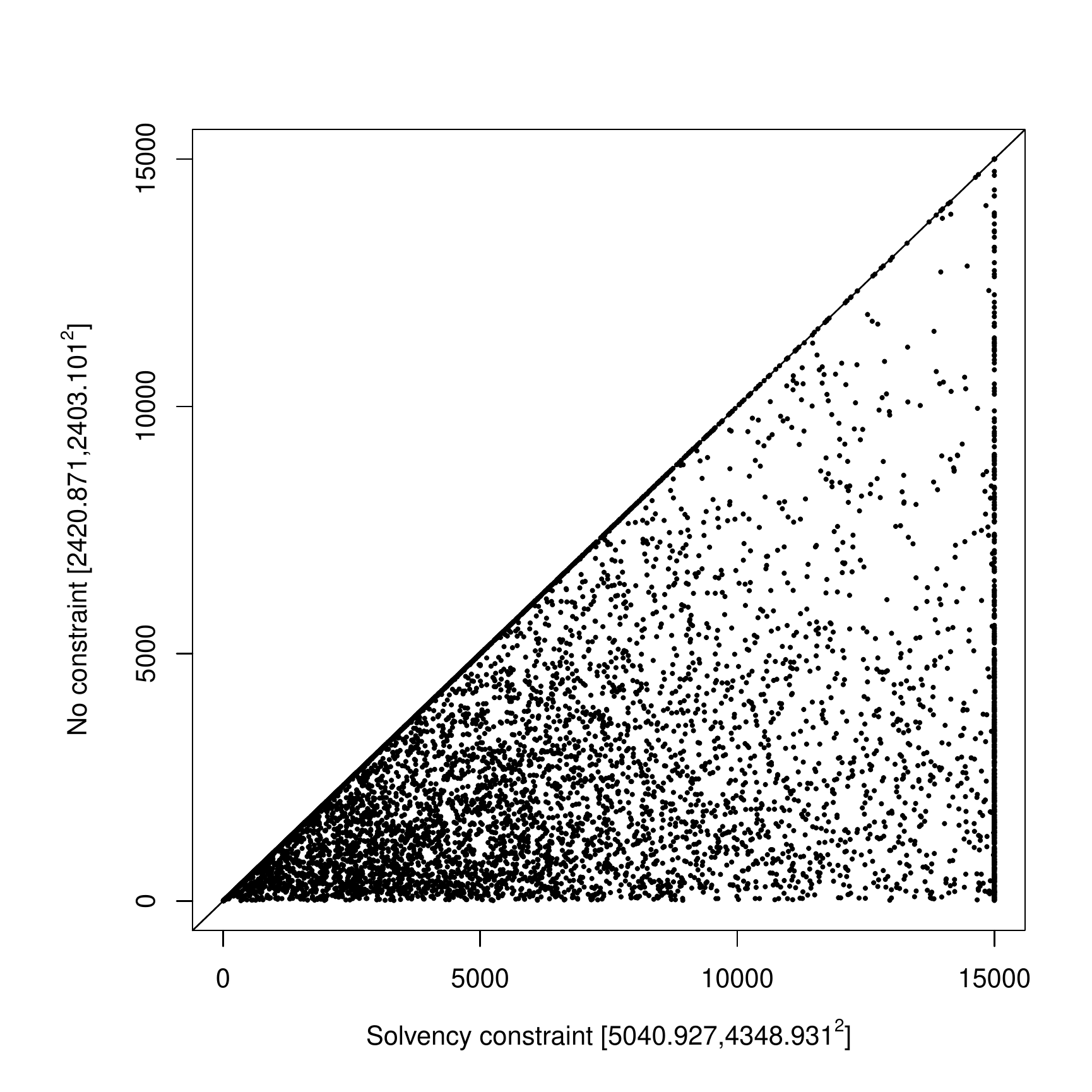}
 \caption{Scatter plots of 10,000 simulations in presence of solvency constraints (horizontal axis) or absence of solvency constraints (vertical axis). \\ 
 The left figure shows pairs of outcomes for the present value of dividends until ruin for each simulation, whereas and the right figure shows pairs of outcomes of the time to ruin (censored at 15,000). The figures shown in-between squared brackets are the mean and variance of the present value of dividends on the left hand side, and time to ruin (censored at 15,000) on the right hand side.}

 \label{F:scatter:simple:constraints}
\end{center}
\end{figure}

\subsection{The impact of the solvency constraint}\label{S::simple:discussion}
  In this section, we consider the impact of the  solvency constraint on the stability of operations. In Figure \ref{F:scatter:simple:constraints}, we compare outcomes of $10{,}000$ simulations (according to an Euler scheme) for the aggregate distributed amount (left), as well as the associated time to bankruptcy (right), when a  solvency constraint is applied (horizontal axis) or not (vertical axis). The simulations were censored at time $T=15{,}000$ (unless declared bankrupt before). \setenceToReferneceParameters{1}
  
  In all scatter plots, values in the top left triangle are those where the outcome in absence of constraints beats that in presence of a solvency constraint, whereas outcomes in the bottom right triangle are those where the constraints beats the base case. 
  
  In terms of dividends, we know that the absence of constraints will lead to a higher expected present value---\emph{on average} (here $0.778>0.774$). However, what the left hand side of Figure \ref{F:scatter:simple:constraints} teaches us is that, when there is a substantial difference, it is in favour of the solvency constraint (dots that are significantly away from the 45 degree line are in the bottom right triangle). Also, the coefficient of variation of the aggregate distribution amount is lower with the solvency constraint than without ($0.214/0.774<0.217/0.778$). 
  

{\color{black} We now turn our attention to right hand side of Figure \ref{F:scatter:simple:constraints}, which focuses on the the impact of the solvency constraint on the time to ruin. It illustrates that the time to ruin in the presence of a solvency constraint clearly dominates the time to ruin in the absence of solvency constraint. This is evidenced by (i) the outcomes are in the bottom triangle (away from the 45 degree line), and (ii) moreover, many trajectories are not ruined yet after $T=15,\!000$ time units when a solvency constraint is applied, whereas the corresponding trajectories re-calculated in absence of solvency constraint were ruined (see the vertical accumulation of dots at $T=15,\!000$). } 
 
  
  Combining both results, one could argue that the solvency constraint in this case is really effective and comes at a relatively small cost. Note that while these results would be quantitatively different with other parameters, conclusions would be qualitatively similar.

\subsection{The impact of of capital injections}\label{S::capitalInjections:discussion}

In this section, we illustrate the impact of introducing capital injections on the dividend strategy and its outcomes.

\begin{figure}[htb]\begin{center}
\includegraphics[width=0.5\textwidth, clip=true, trim = 15mm 0mm 10mm 20mm]{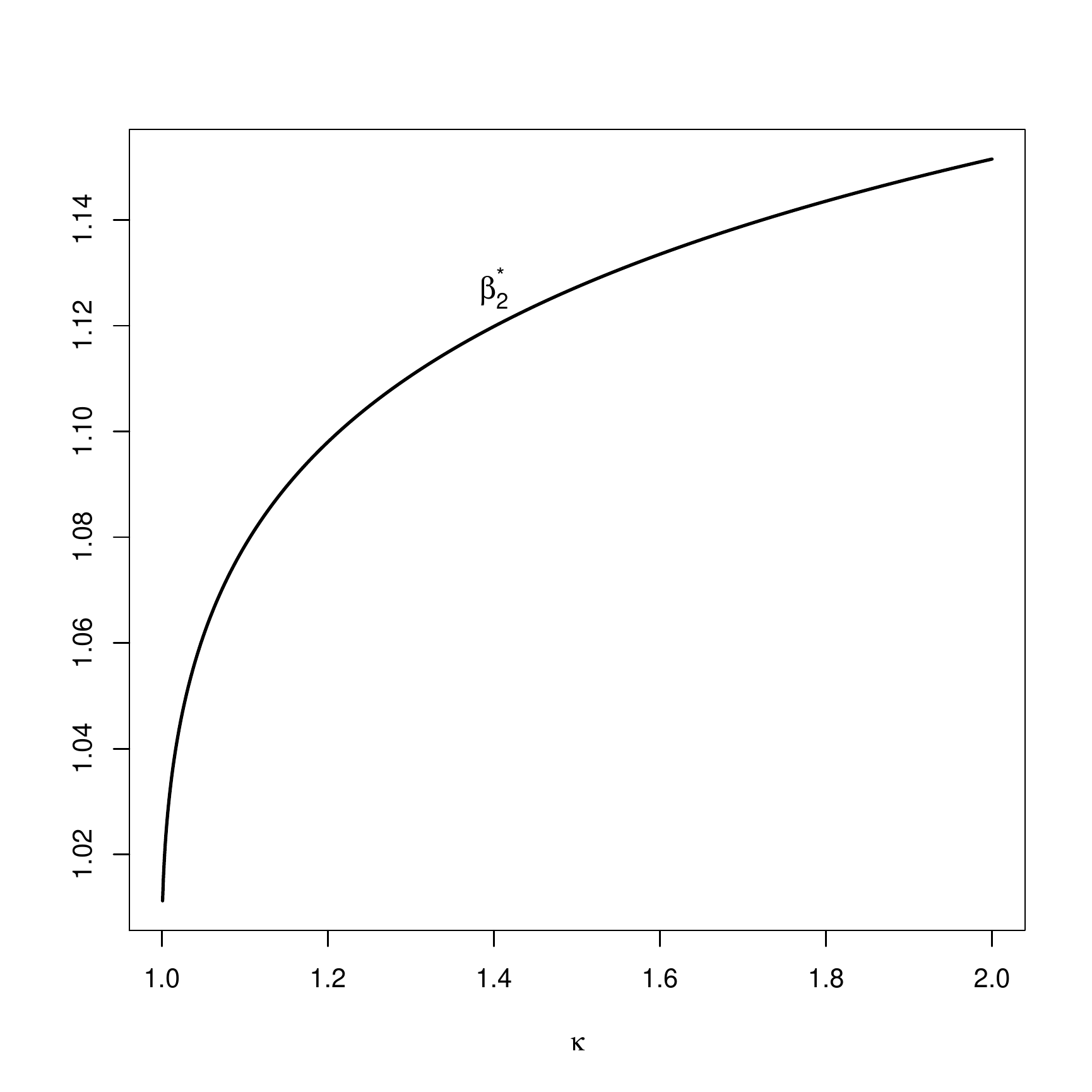}
 \caption{The optimal barriers $\beta_{2}^*$ as a function of $\kappa$. The higher the transaction costs, the higher the optimal dividend barrier level. \\ \setenceToReferneceParameters{1}}
\label{F:CI:optimal:barriers}
\end{center} 
\end{figure}

\subsubsection{The impact of transaction costs on the optimal distribution barrier}
  Whatever the level of transaction costs $\kappa$, the optimal capital injection level $\gamma^*$ will always be $\alpha_0$. However, the optimal distribution barrier will be affected by different values of $\kappa$. This is illustrated in Figure \ref{F:CI:optimal:barriers}. For $\kappa=1$, the optimal distribution barrier is $\alpha_0$ since there is no reason to hold an extra buffer when additional capital comes at no cost. It then increases as $\kappa$ increases. 

\begin{figure}[htb]\begin{center} 
\includegraphics[width=0.6\textwidth, clip=true, trim = 25mm 20mm 35mm 30mm]{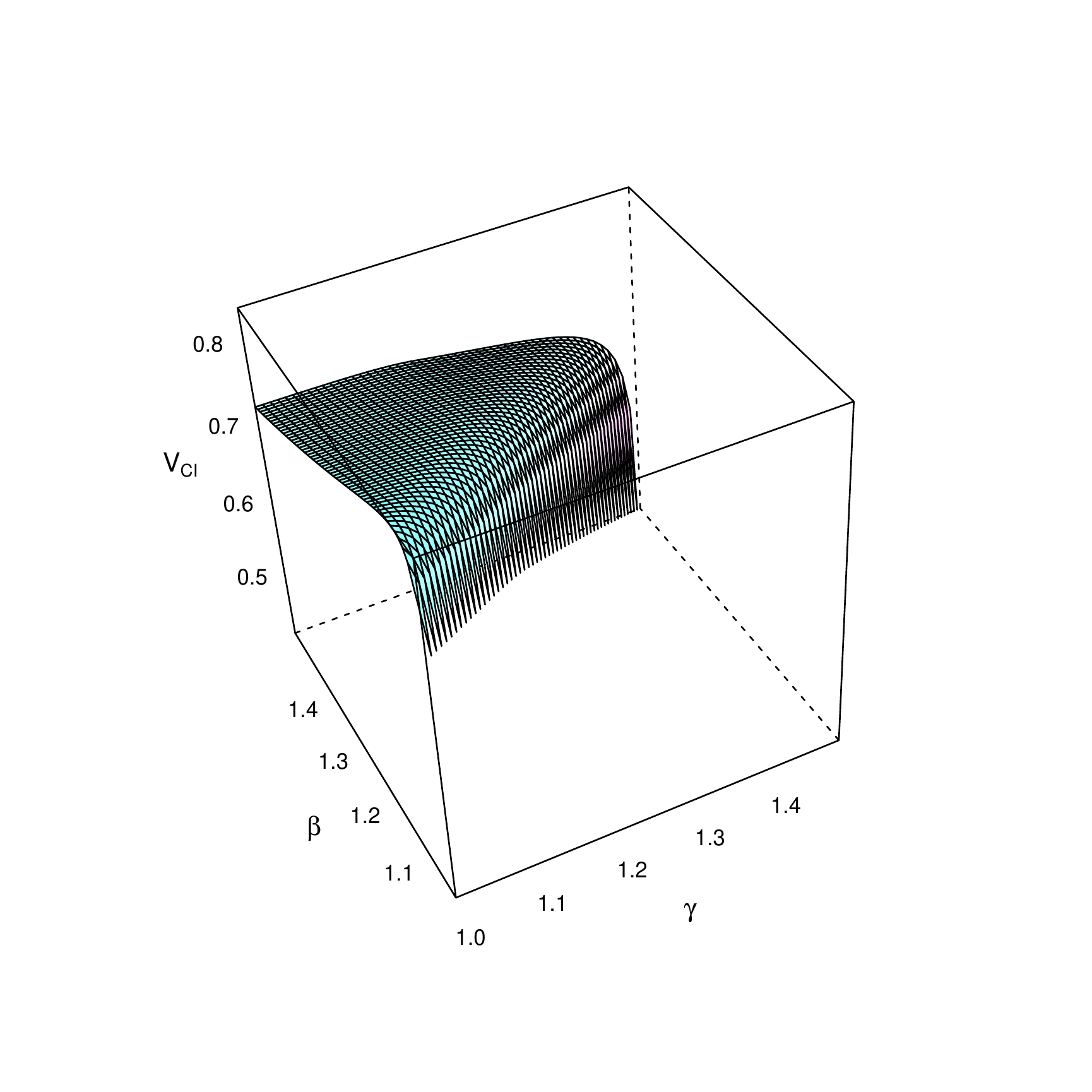}
 \caption{Surface plots of the value functions as functions of $\gamma$ and $\beta$ with $\kappa=1.05$. \\
 \setenceToReferneceParameters{1}
 }
\label{F:CI:surface}
\end{center}\end{figure}

\subsubsection{The impact of the capital injection and distribution barriers on the value function}
  In Figure \ref{F:CI:surface}, we depict surface plots for $\gamma$ and $\beta$ of the value function for the model in the case where capital is injected to prevent ruin. The existence of an optimal level for $\beta$ is obvious, as is the fact that the optimal level for $\gamma$ will always be the ruin level (here $\alpha_0=1$).

\subsubsection{Should capital be injected to prevent ruin?} \label{S_rescue}

In this section, we explore the conjecture spelt out in Remark \ref{R_forced}, which asserts that, in fact, the company should be rescued only if the value function at the capital injection barrier is nonnegative. Under which combination of parameters would rescue be optimal?
  
In Figure \ref{F:CI:valuFunction0}, we plot the values of $\kappa$ which makes the value function exactly equal to $0$ for the optimal choice of the upper barrier. This is because everything else being equal, the cost of injecting capital ($\kappa$) will be the main parameter determining whether capital injections are worthwhile or not.
  
Lower risk levels will allow for higher levels of transaction costs $\kappa$. In fact, very low risk will make it very attractive to rescue the company even though the cost of capital injection is high. 

\begin{figure}[htb]\begin{center}

\includegraphics[width=0.6\textwidth, clip=true, trim = 5mm 0mm 5mm 10mm]{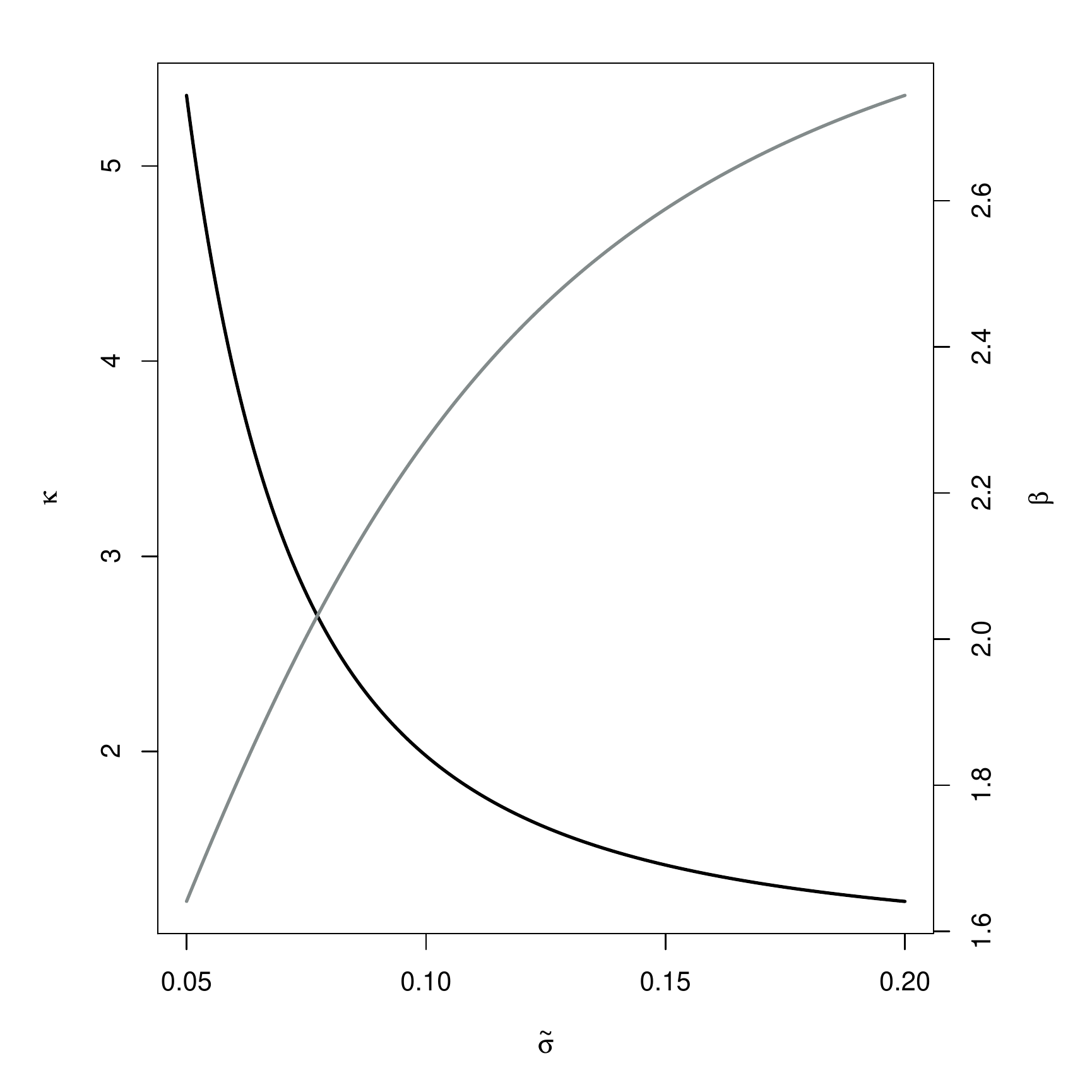}
 \caption{Illustration of which values of $\kappa$ (black line) that makes the value functions equal to $0$ at the funding ratio $\alpha_0$. The corresponding optimal values for $\beta$ are in grey. Assuming the conjecture is correct, this means that if $\kappa$ is higher than the plotted figure, then one should never rescue the company with capital injections (and vice versa).  
 \\
 \setenceToReferneceParameters{2}} 
\label{F:CI:valuFunction0}
\end{center}\end{figure}

\section*{Acknowledgements}

Part of the paper was written while Henriksen was visiting the other authors at the UNSW Business
School 
and the Department of Mathematics and Statistics of the University of Montreal.
Henriksen would like to thank both Universities for their hospitality. Henriksen would also like to thank Griselda Deelstra and Jostein Paulsen for helpful comments on earlier versions of this research, which appeared in his PhD thesis.

This research was supported under Australian Research Council's Linkage (LP130100723, with funding partners Allianz Australia Insurance Ltd, Insurance Australia Group Ltd, and Suncorp Metway Ltd) and Discovery Project (DP200101859) funding schemes. Furthermore, Avanzi and Henriksen acknowledge support from a grant of the Natural Science and Engineering Research Council of Canada (project number RGPIN-2015-04975). The views expressed herein are those of the authors and are not necessarily those of the supporting organisations. 

\section*{References}
\bibliographystyle{elsarticle-harv}
\bibliography{libraries} 

\appendix

\section{Proof of Lemma \ref{lemma:G:simple}}\label{proof:lemma:G:simple}

\begin{proof}[Proof of Lemma \ref{lemma:G:simple}]
We introduce the notation $\tilde{G}(\cdot;\beta)$ by
\begin{align*}\begin{split}
G(x_1,x_2;\beta)&=(x_1+x_2)G\left(\frac{x_1}{x_1+x_2},\frac{x_2}{x_1+x_2};\beta\right)=(x_1+x_2)G\left(y,1-y;\beta\right)\\
&=(x_1+x_2)\tilde{G}\left(y;\beta\right),
\end{split}\end{align*}
where $y=\frac{x_1}{x_1+x_2}$. For reformulation of the HJB equation, we need the following derivatives
\begin{align*}\begin{split}
\pa{x_i}G(x_1,x_2;\beta) \hbox{ and } \frac{\partial^2}{\partial x_i\partial x_j}G(x_1,x_2;\beta), \quad i,j=1,2.
\end{split}\end{align*}
We get
{\allowdisplaybreaks[1]\begin{eqnarray*}
\pa{x_1}G(x_1,x_2;\beta)&=&\pa{x_1}\left(\left(x_1+x_2\right)\tilde{G}\left(\frac{x_1}{x_1+x_2};\beta\right)\right)\\*
&=&\tilde{G}\left(\frac{x_1}{x_1+x_2};\beta\right)+\left(x_1+x_2\right)\tilde{G}'\left(\frac{x_1}{x_1+x_2};\beta\right)\left(\frac{1}{x_1+x_2}-\frac{x_1}{(x_1+x_2)^2}\right)\\
&=&\tilde{G}(y;\beta)+\tilde{G}'(y;\beta)(1-y),\\
\pa{x_2}G(x_1,x_2;\beta)&=&\pa{x_2}\left(\left(x_1+x_2\right)\tilde{G}\left(\frac{x_1}{x_1+x_2};\beta\right)\right)\\*
&=&\tilde{G}\left(\frac{x_1}{x_1+x_2};\beta\right)+\left(x_1+x_2\right)\tilde{G}'\left(\frac{x_1}{x_1+x_2};\beta\right)\left(-\frac{x_1}{(x_1+x_2)^2}\right)\\*
&=&\tilde{G}(y;\beta)+\tilde{G}'(y;\beta)(-y),\\
\paTo{x_1}G(x_1,x_2;\beta)&=&\paTo{x_1}\left(\left(x_1+x_2\right)\tilde{G}\left(\frac{x_1}{x_1+x_2};\beta\right)\right)\\*
&=&\pa{x_1}\left(\tilde{G}\left(\frac{x_1}{x_1+x_2};\beta\right)+\frac{x_2}{x_1+x_2}\tilde{G}'\left(\frac{x_1}{x_1+x_2};\beta\right)\right)\\*
&=&\tilde{G}'\left(\frac{x_1}{x_1+x_2};\beta\right)\left(\frac{1}{x_1+x_2}-\frac{x_1}{(x_1+x_2)^2}-{\frac {x_2}{ \left( x_1+x_2 \right) ^{2}}}\right)\\*
&&+\frac{x_2}{x_1+x_2}\tilde{G}''\left(\frac{x_1}{x_1+x_2};\beta\right)\frac{x_2}{(x_1+x_2)^2}\\
&=&\tilde{G}''(y;\beta)\frac{x_2^2}{(x_1+x_2)^3}.
\end{eqnarray*}}
Likewise,
\begin{align*}\begin{split}
\paTo{x_2}G(x_1,x_2;\beta)=&\tilde{G}''(y;\beta)\frac{x_1^2}{(x_1+x_2)^3},\\
\frac{\partial^2}{\partial x_1\partial x_2}G(x_1,x_2;\beta)=&-\tilde{G}''(y;\beta)\frac{x_1x_2}{(x_1+x_2)^3}.
\end{split}\end{align*}
In total we get that 
\begin{align}\begin{split}\label{HJB:G:tilde}
\frac{\mathscr{A}G(x_1,x_2;\beta)-\delta G(x_1,x_2;\beta)}{x_1+x_2}=&
\frac{1}{2}\left(\sigma_A^2+\sigma_L^2-2\rho\sigma_A\sigma_L\right)y^2(1-y)^2\tilde{G}''(y;\beta)+\left(\mu_A-\mu_L\right)y(1-y)\tilde{G}'(y;\beta)\\
&+\left(\mu_Ay+\mu_L(1-y)-\delta\right)\tilde{G}(y;\beta).
\end{split}\end{align}
We guess that the solution to the right-hand side of equation $(\ref{HJB:G:tilde})$ equal to $0$ has a solution of the form:
\begin{align}\begin{split}
\tilde{G}(y;\beta)=y^\vartheta(1-y)^\varphi. \label{eqt.guess}
\end{split}\end{align}
Inserting this in (\ref{HJB:G:tilde}) gives:
\begin{align*}\begin{split}
&\frac{\mathscr{A}G(x_1,x_2;\beta)-\delta G(x_1,x_2;\beta)}{x_1+x_2}\\
=&\frac{1}{2}\left(\sigma_A^2+\sigma_L^2-2\rho\sigma_A\sigma_L\right)y^2(1-y)^2\\
&\times\left({y}^{\vartheta-2}\left(\vartheta^2-\vartheta\right) \left( 1-y \right) ^{\varphi}-2\,{y}^{\vartheta-1}\vartheta \left(
1-y \right) ^{\varphi-1}\varphi+{y}^{\vartheta} \left( 1-y \right) ^{\varphi-2}
\left({\varphi}^{2}-\varphi\right)\right) \\
&+\left(\mu_A-\mu_L\right)y(1-y)\left(\vartheta y^{\vartheta-1}(1-y)^{\varphi}-\varphi y^{\vartheta}(1-y)^{\varphi-1}\right)+\left(\mu_Ay+\mu_L(1-y)-\delta\right)y^\vartheta(1-y)^\varphi.
\end{split}\end{align*}
Dividing the above equation with $y^\vartheta(1-y)^\varphi=\tilde{G}(y)$ and setting 
\begin{eqnarray}\label{eqn:sigma:tilde}
\tilde{\sigma}^2=\sigma_A^2+\sigma_L^2-2\rho\sigma_A\sigma_L
\end{eqnarray}
yields that the right-hand side is equal to
\begin{align}\begin{split}\label{quadratic:equation:combined}
&{\frac{1}{2}\left(\vartheta^2-\vartheta\right)\tilde{\sigma}^2+\vartheta(\mu_A-\mu_L)+\mu_L-\delta}+(\vartheta+\varphi-1)\left((-\vartheta\tilde{\sigma}^2-\mu_A+\mu_L)y +\frac{1}{2}\tilde{\sigma}^2(\vartheta+\varphi)y^2\right).
\end{split}\end{align}
Setting the part of (\ref{quadratic:equation:combined}) not depending on $y$ equal to $0$ gives us a quadratic equation for \begin{align}\begin{split} \label{quadratic:equation}
\frac{1}{2}\tilde{\sigma}^2\vartheta^2+\left(\mu_A-\mu_L-\frac{1}{2}\tilde{\sigma}^2\right)\vartheta +\mu_L-\delta=0,
\end{split}\end{align}
and setting the last term of (\ref{quadratic:equation:combined}) equal to $0$ we get that $\varphi=1-\vartheta$. We denote by $\zeta_1$ and $\zeta_2$ the two solutions to the quadratic equation. The solutions are given by

\begin{eqnarray}
&&\frac{-\left(\mu_A-\mu_L-\frac{1}{2}\tilde{\sigma}^2\right)\pm\sqrt{\left(\mu_A-\mu_L-\frac{1}{2}\tilde{\sigma}^2\right)^2-4\frac{1}{2}\tilde{\sigma}^2\left(\mu_L-\delta\right)}}{\tilde{\sigma}^2}\nonumber\\
&&=\frac{\frac{1}{2}\tilde{\sigma}^2-\left(\mu_A-\mu_L\right)\pm\sqrt{\frac{1}{4}\tilde{\sigma}^4+\left(\mu_A-\mu_L\right)^2-\left(\mu_A-\mu_L\right)\tilde{\sigma}^2-2\tilde{\sigma}^2\left(\mu_L-\delta\right)}}{\tilde{\sigma}^2}\nonumber\\
&&=\frac{\frac{1}{2}\tilde{\sigma}^2-\left(\mu_A-\mu_L\right)\pm\sqrt{\frac{1}{4}\tilde{\sigma}^4+\left(\mu_A-\mu_L\right)^2-\tilde{\sigma}^2\left(\mu_A+\mu_L-2\delta\right)}}{\tilde{\sigma}^2}.\label{sol:quad:eqn}
\end{eqnarray}
Because the coefficient of the quadratic term of (\ref{quadratic:equation}), $\frac{1}{2}\tilde{\sigma}^2$, is greater than $0$ and because the left-hand side of (\ref{quadratic:equation}) is negative for $\vartheta=0$ by (\ref{drift:assumption}) and (\ref{drift:assumption:2}) the quadratic equation (\ref{quadratic:equation}) has a positive solution, which we denote $\zeta_2$, and negative solution, which we denote $\zeta_1$. Because the left-hand side of (\ref{quadratic:equation}) is equal to $\mu_A-\delta<0$ for $\vartheta=1$, we get that $\zeta_2>1$.

  By using that $\tilde{G}(y;\beta)=G(y,1-y;\beta)$ we get that a general solution is given by
  \begin{align}\begin{split}\label{eqn:general:representation:value:function}
  G(x_1,x_2;\beta)=\vartheta_1{x_1}^{\zeta_1}{x_2}^{1-\zeta_1}+\vartheta_2{x_1}^{\zeta_2}{x_2}^{1-\zeta_2},
  \end{split}\end{align} 
  which is equivalent to $x_2\left(\vartheta_1{\left(\frac{x_1}{x_2}\right)}^{\zeta_1}+\vartheta_2{\left(\frac{x_1}{x_2}\right)}^{\zeta_2}\right)$. Solving $\vartheta_1$ and $\vartheta_2$ using (\ref{BC:D0:B}) and (\ref{BC:D1:B}), we obtain (\ref{equation:trial:G:B}).
 
  Note the solution to $(\mathscr{A}-\delta)\tilde{G}(y;\beta)=0$ is unique, given the 2 boundary conditions (\ref{BC:D0:B}) and (\ref{BC:D1:B}). This shows that our guess (\ref{eqt.guess}) is correct.
\end{proof}

\section{Parameter values}\label{A_B}
\begin{table}[H]\begin{center}
\begin{tabular}{|r|r|r|r|r|r|r|r|r|r|r|r|r|}
\hline
No.&\cc{$\rho$} &\cc{$\delta$} & \cc{$\mu_A$} & \cc{$\mu_L$} & \cc{$\sigma_A$} & \cc{$\sigma_L$}& \cc{$\alpha_0$}&\cc{$\alpha_1$}&\cc{$\alpha_2$}&\cc{$\kappa$}&\cc{$A_0$}&\cc{$L_0$}\\ \hline
1& 0.5 & 0.055 & 0.05 & 0.04& 0.03& 0.01&1&1.3&1.35&1.05& 1.2&1\\\hline
2& 0.5 & 0.055 & 0.05 & 0.04& - & 0.01&1&2.5&-&-& 1&1\\\hline 
\end{tabular}
\caption{Parameter values for numerical illustrations.}
\label{tab:parameter:exp:v2}
\end{center}\end{table}

\end{document}